\documentclass{amsart}
\usepackage{todonotes}
\usepackage{amsmath,amssymb,dsfont}
\usepackage{graphicx,color}
\usepackage[normalem]{ulem} 
\newtheorem{theorem}{Theorem}[section]
\newtheorem{lemma}[theorem]{Lemma}
\newtheorem{proposition}[theorem]{Proposition}
\newtheorem{corollary}[theorem]{Corollary}

\theoremstyle{definition}
\newtheorem{definition}[theorem]{Definition}

\newtheorem{remark}[theorem]{Remark}

\numberwithin{equation}{section}

\newcommand{\rd}{{\,\rm d}}
\newcommand{\e}{{\rm e}}

\newcommand{\N}{{\mathbb N}}

\newcommand{\R}{{\mathbb R}}
\newcommand{\C}{{\mathbb C}}
\newcommand{\Z}{{\mathbb Z}}

\newcommand\beq{\begin{equation}}
\newcommand\eeq{\end{equation}}

\newcommand{\dist}{\mathrm{dist}}
\newcommand\re{\mathrm{Re}}
\newcommand\im{\mathrm{Im}}
\newcommand\I{\mathrm{i}}
\newcommand{\beqnt}{\begin{equation*}}
\newcommand{\eeqnt}{\end{equation*}}
\newcommand{\set}[2]{\{#1 : #2 \}}

\newcommand{\sgn}{\operatorname{sgn}}

\DeclareMathOperator{\supp}{supp}
\DeclareMathOperator{\dom}{dom}
\DeclareMathOperator{\Lip}{Lip}
\DeclareMathOperator{\Op}{Op}

\begin{document}

\title[Resolvent Estimates for magnetic Schr\"odinger operators]{$L^p$ resolvent estimates for magnetic Schr\"odinger operators with unbounded background fields}

\author{Jean-Claude Cuenin}
\address{Mathematisches Institut, Ludwig-Maximilians-Universit\"at M\"unchen, 80333 Munich, Germany}
\email{cuenin@math.lmu.de}

\author{Carlos E.\ Kenig}
\address{Department of Mathematics, University of Chicago, Chicago, IL 60637, USA}
\email{cek@math.uchicago.edu}

\begin{abstract}
We prove $L^p$ and smoothing estimates for the resolvent of magnetic Schr\"odinger operators. We allow electromagnetic potentials that are small perturbations of a smooth, but possibly unbounded background potential. As an application, we prove an estimate on the location of eigenvalues of magnetic Schr\"odinger and Pauli operators with complex electromagnetic potentials. 
\end{abstract}

\maketitle

\section{Introduction}

Resolvent estimates for Schr\"odinger operators play a decisive role in numerous areas in spectral and scattering theory, as well as partial differential equations. In particular, resolvent estimates which are uniform in the spectral parameter are intimately connected with dispersive and smoothing estimates for the corresponding (time-dependent) Schr\"odinger equation, as observed by Kato \cite{Kato1965}. 

As a general rule, resolvent estimates that hold up to the spectrum (usually called a limiting absorption principle) are   
associated with global in time Strichartz and smoothing estimates for the Schr\"odinger flow. Results in this category are usually obtained by considering a decaying electromagnetic potential as a perturbation of the free Laplacian, see for example \cite{GeorgievEtAl2007} for small perturbations and \cite{ErdoganEtAl2008,ErdoganEtAl2009,FanelliVega2009,D'AnconaEtAl2010,Garcia2011,Garcia2015} for large perturbations.

On the other hand, resolvent estimates that are uniform only up to a $\mathcal{O}(1)$ distance to the spectrum are associated with local in time estimates. This is usually due to the presence of eigenvalues or resonances that prevent the dispersion of the flow. Potentials in this situation are usually unbounded. Prominent examples here are the harmonic oscillator (quadratic electric potential) and the constant magnetic field (linear vector potential). 
We mention \cite{Fujiwara1980,Yajima1991,YajimaZhang2004,Doi2005,RobbianoZuily2008,D'AnconaFanelli2009} for estimates involving unbounded potentials. 

There is a big gap in the regularity and decay conditions for the electromagnetic potential 
between the two scenarios. In the first, the potentials can usually be quite rough, but have sufficient decay at infinity.
In the second case, unbounded potentials are allowed but they are usually assumed to be smooth. Very little is known in the intermediate case. Our resolvent estimate is a step in this direction.


The paper is organized as follows. In Section \ref{section Assumptions and main result} we state the assumptions on the potentials and the resolvent estimate in the simplest case, with a uniform bound with respect to the spectral parameter. In Section \ref{section Lq'Lq and smoothing estimates} we prove the resolvent estimate for the unperturbed operator. In Section \ref{section Proof main theorem, general case} we use a perturbative argument to prove the estimate in the general case.
In the final Section \ref{section complex-valued potentials} we state a more precise version of the resolvent estimate and give an application to eigenvalue bounds for Schr\"odinger operators with complex-valued potentials.\\

\underline{{\bf Notation}} 
\begin{itemize}
\item $\langle x\rangle=(1+|x|^2)^{1/2}$.
\item $\langle D\rangle$ is the Fourier multiplier with symbol $\langle\xi\rangle$.
\item $X=(x,\xi)\in\R^{2n}$. 
\item $e_s(X):=\langle X\rangle^s$ and $E_{s}:=e_s^W(x,D)$; see also Appendix \ref{Appendix psdos}.
\item $\mathcal{D}(\R^n)=C_c^{\infty}(\R^n)$, and $\mathcal{D}'(\R^n)$ is the space of distributions.
\item $\mathcal{S}(\R^n)$ is the Schwartz space, and $\mathcal{S}'(\R^n)$ is the space of tempered distributions.
\item $\mathcal{B}(X,Y)$ is the space of bounded linear operators between Banach spaces $X$ and $Y$.
\item $A\lesssim B$ if there exist a constant $C>0$ (depending only on fixed quantities) such that $A\leq CB$.
\item $\langle u,v\rangle=\int_{\R^n}u(x)\overline{v}(x)\rd x$ for $u,v\in \mathcal{S}(\R^n)$. 
\item If $X$ is a Banach space densely and continuously embedded in $L^2(\R^n)$, we identify $L^2(\R^n)$ with a dense subspace of $X'$. Thus, the duality pairing $\langle\cdot,\cdot\rangle_{X,X'}$ extends the $L^2$-scalar product $\langle\cdot,\cdot\rangle$. This is meant when we write $X\subset L^2(\R^n)\subset X'$.
\item $\sigma(P)$ is the spectrum of $P$.
\item $\dom(P)$ is the domain of $P$.
\end{itemize} 

\section{Assumptions and main result}\label{section Assumptions and main result}

We consider the Schr\"odinger operator
\begin{align}\label{P}
P=(-\I\nabla+A(x))^2+V(x),\quad \dom(P)=\mathcal{D}(\R^n)\subset L^2(\R^n),\quad n\geq 2.
\end{align}
Here, $A:\R^n\to \R^n$ is the vector potential and $V:\R^n\to \R$ is the electric potential. 
In the following, $\epsilon>0$ is a yet undetermined constant that will later be chosen sufficiently small

\vspace{10pt}

\underline{\bf Assumptions on the potentials}
Let $A=A_0+ A_1$ and $V=V_0+W+ V_1$ and assume that the following assumptions hold.
\begin{enumerate}
\item[(A1)] $A_0\in C^{\infty}(\R^n,\R^n)$ and for every $\alpha\in\N^n$, $|\alpha|\geq 1$, there exist constants $C_{\alpha},\epsilon_{\alpha}>0$ such that
\begin{align}
|\partial_x^{\alpha} A_0(x)|\leq C_{\alpha},\quad|\partial^{\alpha}B_0(x)|\leq C_{\alpha}\langle x\rangle^{-1-\epsilon_{\alpha}},\quad x\in\R^n.\label{assumptions on A_0}
\end{align}
Here, $B_0=(B_{0,j,k})_{j,k=1}^n$ is the magnetic field, i.e.
\begin{align*}
B_{0,j,k}(x)=\partial_jA_{0,k}(x)-\partial_kA_{0,j}(x).
\end{align*}
\item[(A2)] $V_0\in C^{\infty}(\R^n,\R)$ and for every $\alpha\in\N^n$, $|\alpha|\geq 2$, there exist constants $C_{\alpha}>0$ such that
\begin{align}
|\partial_x^{\alpha} V_0(x)|&\leq C_{\alpha},\quad x\in\R^n.\label{assumptions on V_0}
\end{align}
\item[(A3)] $W\in L^{\infty}(\R^n,\R)$.
\item[(A4)] $A_1\in L^{\infty}(\R^n,\R)$ and there exists $\delta>0$ such that
\begin{align*}
|A_1(x)|\lesssim \epsilon\langle x\rangle^{-1-\delta}\quad\mbox{for almost every }x\in\R^n.
\end{align*}
Moreover, assume that one of the following additional assumptions holds: 
\begin{enumerate}
\item[(A4a)] $A_1\in \Lip(\R^n,\R^n)$ and
\begin{align*}
|\nabla A_1(x)|\lesssim \epsilon\langle x\rangle^{-1-\delta}\quad\mbox{for almost every }x\in\R^n.
\end{align*}
\item[(A4b)] There exists $\delta'\in (0,\delta)$ such that $\langle x\rangle^{1+\delta'} A_1\in \dot{W}^{\frac{1}{2},2n}(\R^n;\R^n)$, with $\|\langle x\rangle^{1+\delta'}A_1\|_{\dot{W}^{\frac{1}{2},2n}}\lesssim \epsilon$.
\end{enumerate}
\item[(A5)] Assume that $V_1\in L^{r}(\R^n,\R)$, with $\|V_1\|_{L^r}\lesssim \epsilon$, for some $r\in(1,\infty]$ if $n=2$ and $r \in[n/2,\infty]$ if $n\geq 3$.
\end{enumerate}

\begin{remark}
We can relax the assumption \eqref{assumptions on V_0} in the same way as in \cite{KochTataru2005Hermite}.
\begin{enumerate}
\item[(A2')] $V_0\in C^{2}(\R^n,\R)$ and for every $\alpha\in\N^n$, $|\alpha|=2$, there exist constants $C_{\alpha}$ such that
\begin{align}
|\partial_x^{\alpha} V_0(x)|\leq C_{\alpha},\quad x\in\R^n.\label{assumptions on V_0'}
\end{align}
\end{enumerate}
To see this, we decompose $V_0$ into its low-frequency and its hight-frequency part, $V_0=V_0^{\rm low}+V_0^{\rm high}$. Here, $V_0^{\rm low}:=\chi(D)V_0$, 
where $\chi\in \mathcal{D}(\R^n)$ is supported in $B(0,2)$ and $\chi\equiv 1$ in $B(0,1)$.
Then, by Bernstein inequalities $V_0^{\rm low}$ satisfies (A2). On the other hand, $V_0^{\rm high}\in L^{\infty}(\R^n)$, so this term can be absorbed into $W$.

It would be natural to also try to relax the smoothness assumption on $A_0$ to $C^1(\R^n,\R^n)$. However, if we just split such an $A_0$ into high and low frequency parts, then $A_0^{\rm high}$ will have no decay, and thus it cannot be absorbed into the perturbative part. Moreover, even if it could be absorbed, then it would not be small. 
\end{remark}

\begin{remark}
Assumption (A4b) was used in \cite{ErdoganEtAl2009}, where it was also remarked that a condition similar to (A4a), but with $|\nabla A_1(x)|\lesssim\langle x\rangle^{-2-\delta}$, would imply (A4b). Here we state both conditions, because neither is weaker than the other. There is an obvious trade-off between decay and regularity. 
\end{remark}


We will consider $P$ as a small perturbation of the Schr\"odinger operator
\begin{align}\label{P_0 tilde}
\widetilde{P}_0=(-\I\nabla+A_0(x))^2+\widetilde{V}_0(x),\quad \dom(\widetilde{P}_0)=\mathcal{D}(\R^n)
\end{align}
where $\widetilde{V}:=V_0+W$.
In the case $W=0$, we also write
\begin{align}\label{P_0}
P_0=(-\I\nabla+A_0(x))^2+V_0(x),\quad \dom(P_0)=\mathcal{D}(\R^n).
\end{align}

Our resolvent estimate involves the following spaces. Let $X$ be the completion of $\mathcal{D}(\R^n)$ with respect to the norm 
\begin{center}
\fbox{$\|u\|_X:=\|u\|_{L^2}+\|\langle x\rangle^{-\frac{1+\mu}{2}}E_{1/2} u\|_{L^2}+\|u\|_{L^{q}}
$}
\end{center}
where $0<\mu\leq \delta$ is fixed. Then its topological dual $X'$ is the space of distributions $f\in \mathcal{D}'(\R^n)$ such that the norm
\begin{center}
\fbox{$\|f\|_{X'}:=\inf_{f=f_1+f_2+f_3}\left(\|f_1\|_{L^2}+\|\langle x\rangle^{\frac{1+\mu}{2}}E_{-1/2} f_2\|_{L^2}+\|f_3\|_{L^{q'}}\right)$}
\end{center}
is finite. Here, $q=2r'$, i.e.
\begin{align*}
\begin{cases}
q\in [2,\infty) \quad &\mbox{if  } n=2,\\
q\in [2,2n/(n-2)]\quad &\mbox{if  } n\geq 3,
\end{cases}
\end{align*}
Our main result is the following theorem.

\begin{theorem}\label{thm resolvent estimate X'X}
Assume that $A$, $V$ satisfy Assumptions (A1)--(A5). Moreover, let $a>0$ be fixed. Then there exists $\epsilon_0>0$ such that for all $\epsilon<\epsilon_0$, we have the estimate
\begin{align}\label{eq. resolvent estimate main theorem}
\|u\|_{X}\leq C
\|(P-z)u\|_{X'}
\end{align}
for all $z\in\C$ with $|\im z|\geq a$ and for all $u\in \mathcal{D}(\R^n)$. The constants $\epsilon_0,C$ depend on $n$, $q$, $\mu$, $\delta$, $\delta'$, $a$, $\|W\|_{L^{\infty}}$ and on finitely many seminorms $C_{\alpha}$ in \eqref{assumptions on A_0} and \eqref{assumptions on V_0}. 
\end{theorem}

\begin{remark}
In Section \ref{section complex-valued potentials} we will state a more precise version of the estimate~\eqref{eq. resolvent estimate main theorem} that takes into account the dependence of $C$ on $|\im z|$ for large values of $|\im z|$. We will also allow $V_1$, $A_1$ to be complex-valued.
\end{remark}

\begin{remark}
Although we always assume that $a>0$ is fixed, it can be seen by inspection of the proof that the constant in \eqref{eq. resolvent estimate main theorem} is $C=\mathcal{O}(a^{-1})$ as $a\to 0$. Moreover, one could replace the condition $|\im z|\geq a>0$ by the similar condition $\dist(z,\sigma(\widetilde{P}_0))\geq a'>0$ or, a fortiori, $\dist(z,\sigma(P_0))\geq a'+\|W\|_{L^{\infty}}$.  
\end{remark}

\begin{remark}
It is instructive to consider the example $P=-\Delta$ (i.e.\ $A=V=0$). Assume that $\re z>0$. Then, by a scaling argument, the estimate \eqref{eq. resolvent estimate main theorem} implies that
\begin{align}\label{uniform Sobolev Laplacian}
\|u\|_{L^q}\leq C \epsilon^{\frac{n}{2}(1/q'-1/q)-1}\|(-\Delta-1\pm \I\epsilon)u\|_{L^{q'}}
\end{align}
for all $\epsilon>0$. In the case $1/q-1/q'=2/n$, this is a special case of the uniform Sobolev inequality in \cite{KenigRuizSogge1987}. 
\end{remark}

\begin{remark}
It is clear that if $A,V\neq 0$, a uniform estimate of the form \eqref{uniform Sobolev Laplacian} for $1/q-1/q'=2/n$ cannot hold in general, due to the possible presence of eigenvalues.
\end{remark}

\section{$L^{q'}\to L^{q}$ and smoothing estimates}\label{section Lq'Lq and smoothing estimates}

In this section, we will establish the $L^{q'}\to L^q$ and the smoothing estimate for $P_0$. These will be the main ingredients in the proof of Theorem \ref{thm resolvent estimate X'X}. We follow the general approach of Koch and Tataru in \cite[Section 4]{KochTataru2009}, where a version of the resolvent estimate \eqref{eq. resolvent estimate main theorem} was proved for the Hermite operator. In fact, the bound~\eqref{eq. resolvent estimate main theorem} follows in this special case by combining Proposition~4.2, Proposition~4.6 and Proposition~4.7 in \cite{KochTataru2009}.

We start with the $L^{q'}\to L^q$ estimate. The proof of Theorem \ref{thm LpLp'} below follows the same arguments as that of Proposition 4.6 in \cite{KochTataru2009} for the Hermite operator, see also \cite[Section 2]{KochTataru2005Hermite}. Note, however, that the symbol of $P_0$ does not satisfy the bounds (3) in \cite{KochTataru2005Hermite}, which implies the short-time dispersive estimate for the propagator in their case. Here, we use results of Yajima \cite{Yajima1991}, see also \cite{MR2172692}.

\begin{theorem}[$L^{q'}\to L^q$ estimate]\label{thm LpLp'}
Assume that $A_0$, $V_0$ satisfy Assumptions (A1)--(A2). Let $q\in[2,\infty)$ if $n=2$ and $q\in [2,2n/(n-2)]$ if $n\geq 3$, and let $a>0$ be fixed.
Then there exists $C_0>0$ such that for all $z\in\C$ with $|\im z|\geq a$ and for all $u\in \mathcal{D}(\R^n)$, we have the estimate
\begin{align}\label{eq. Lq'Lq in thm}
|\im z|^{1/2}\|u\|_{L^2}+\|u\|_{L^q}\leq C_0
\|(P_0-z)u\|_{L^{q'}}.
\end{align}
The constant $C_0$ depends on $n$, $q$, $a$, and on finitely many seminorms $C_{\alpha}$ in \eqref{assumptions on A_0} and \eqref{assumptions on V_0}. 
\end{theorem}

\begin{proof}
By \cite[Theorem 6]{Yajima1991}, $P_0$ is essentially selfadjoint.
By abuse of notation we continue to denote the selfadjoint extension of $P_0$ by the same symbol. It then follows that $\sigma(P_0)\subset\R$ and 
\begin{align}\label{trivial L2L2 estimate with imaginary part}
\|(P_0-z)^{-1}\|_{\mathcal{B}(L^2)}\leq \frac{1}{|\im z|},
\end{align}
where $(P_0-z)^{-1}$ is the $L^2$-resolvent.
By \cite[Theorem 4]{Yajima1991}, we have the short-time dispersive estimate
\begin{align}\label{dispersive estimate L}
\|\e^{\I tP_0}\|_{\mathcal{B}(L^1,L^{\infty})}\lesssim |t|^{-n/2},\quad t\leq T\ll 1.
\end{align}
By now standard abstract arguments \cite{KeelTao1998}, the full range of Strichartz estimates for the Schr\"odinger equation holds: If $(p_1,q_1)$, $(p_2,q_2)$ are sharp Schr\"odinger-admissible, i.e. if
\begin{align}\label{Strichartz admissible}
\frac{2}{p_i}+\frac{n}{q_i}=\frac{n}{2},\quad p_i\in[2,\infty],\quad q_i\in[2,2n/(n-2)],\quad (n,p_i,q_i)\neq (2,2,\infty),
\end{align}
and if $u$ satisfies the Schr\"odinger equation for $P_0$,
\begin{align}\label{Schroedinger equation for L}
\I\partial_t u-P_0 u=f,\quad u|_{t=0}=u_0,
\end{align}
then
\begin{align}\label{Strichartz estimates}
\|u\|_{L^{p_1}_t([0,T];L^{q_1}_x)}\lesssim \|u_0\|_{L^2_x}+\|f\|_{L^{p_2'}_t([0,T];L^{q_2'}_x)}.
\end{align}
For the non-endpoint case $q\neq 2n/(n-2)$, this also follows from \cite[Theorem 1]{Yajima1991}.
We now fix a pair $(p,q)$ as in \eqref{Strichartz admissible} and with $q$ as in the theorem for the rest of the proof.
Let $\Pi_{[k,k+1]}$ be the spectral projection of $P_0$ onto the interval $[k,k+1]$, $k\in\Z$. For any $u\in L^2(\R^n)$, the function $v(x,t):=\e^{-\I tk}\Pi_{[k,k+1]}u(x)$ satisfies the estimate 
\begin{align}\label{eq. energy estimate}
\|(\I\partial_t-P_0)v\|_{L^{\infty}_t([0,T];L^2_x)}\leq \|u\|_{L^2_x}.
\end{align}
Applying the Strichartz estimates \eqref{Strichartz estimates} with $(p_1,q_1)=(p,q)$ and $(p_2,q_2)=(\infty,2)$ to $v(x,t)$, it follows that
\begin{align}\label{O(1) bound spectral projections}
\|\Pi_{[k,k+1]}u\|_{L^{q}}\lesssim \|u\|_{L^2}.
\end{align}
This argument can be found in \cite{KochTataru2005Hermite}, see Corollary 2.3 there.
Since $\Pi_{[k,k+1]}^*=\Pi_{[k,k+1]}=\Pi_{[k,k+1]}^2$, the dual as well as the $TT^*$ version of \eqref{O(1) bound spectral projections} yield
\begin{align}\label{dual and TT* version of O(1) bound spectral projections}
\|\Pi_{[k,k+1]}u\|_{L^2}\lesssim \|u\|_{L^{q'}},\quad\|\Pi_{[k,k+1]}u\|_{L^q}\lesssim \|u\|_{L^{q'}}.
\end{align}
Using the first inequality in \eqref{dual and TT* version of O(1) bound spectral projections}, orthogonality of the spectral projections and the spectral theorem, we see that for any $u\in L^2(\R^n)\cap L^{q'}(\R^n)$, we have
\begin{equation}\begin{split}
\|(P_0-z)^{-1}u\|_{L^2}^2&=\sum_{k\in\Z}\|(P_0-z)^{-1}\Pi_{[k,k+1]}u\|_{L^2}^2\\
&\leq \sum_{k\in\Z}\sup_{\lambda\in[k,k+1]}|\lambda-z|^{-2}\|\Pi_{[k,k+1]}u\|_{L^2}^2
\lesssim  |\im z|^{-1}\|u\|_{L^{q'}}^2.
\end{split}
\end{equation}
Here, we estimated the sum by
\begin{align*}
\sum_{k\in\Z}\sup_{\lambda\in[k,k+1]}|\lambda-z|^{-2}\leq \sum_{|k-\lfloor{\re z}\rfloor|\leq 3}\frac{1}{a|\im z|}+\sum_{|k-\lfloor{\re z}\rfloor|> 3}\frac{1}{(\lambda-\re z)^2+(\im z)^2}
\end{align*}
where $\lfloor{\re z}\rfloor$ is the integer part of $\re z$.
Setting $k':=k-\lfloor{\re z}\rfloor\in\Z$ in the second sum, we have $|k'|>3$, and
\begin{align*}
|\lambda-\re z|\geq |k'|-|\lambda-k|-|\re z-\lfloor{\re z}\rfloor|\geq |k'|-2>1. 
\end{align*}
Changing variables $k\to k'$, it follows that the second sum is $\mathcal{O}(|\im z|^{-1})$.
By a density argument, we have
\begin{align}\label{eq. Lq'L2}
\|(P_0-z)^{-1}\|_{\mathcal{B}(L^{q'},L^2)}\lesssim |\im z|^{-1/2}.
\end{align}
This proves the first half of \eqref{eq. Lq'Lq in thm}.

We now apply the Strichartz estimates \eqref{Strichartz estimates} to 
$v(x,t)=\e^{-\I tz}u(x)$, assuming that $\im z<0$. Otherwise, we choose $v(x,t)=\e^{\I tz}u(x)$. Note that $v$ satisfies the Schr\"odinger equation \eqref{Schroedinger equation for L} with
\begin{align*}
f(x,t)=\e^{-\I tz}(z-P_0)u(x).
\end{align*}
Applying the Strichartz estimates \eqref{Strichartz estimates} with $(p_1,q_1)=(p_2,q_2)=(p,q)$, 
we obtain
\begin{align}\label{almost resolvent estimate with Strichrtz}
\|u\|_{L^{q}}\lesssim \|u\|_{L^2}+\|(P_0-z)u\|_{L^{q'}}.
\end{align}
Combining \eqref{eq. Lq'L2} and \eqref{almost resolvent estimate with Strichrtz},
we arrive at
\begin{align*}
\|u\|_{L^{q}}\lesssim \|(P_0-z)u\|_{L^{q'}}.
\end{align*}
This proves the second  half of \eqref{eq. Lq'Lq in thm}.
\end{proof}

We now establish the smoothing estimate for $P_0$. We follow Doi \cite{Doi2005}, who considered the corresponding smoothing estimate for the propagator, in a more general situation (time-dependent potentials and non-trivial metric) than we do here. One way to prove the smoothing estimate would be to appeal to the corresponding smoothing estimate for the propagator $\e^{\I tP_0}$ \cite[Theorem 2.8]{Doi2005} and use the same strategy as in the proof of Theorem \ref{thm LpLp'} to deduce the resolvent smoothing estimate. For the sake of clarity, we decided to provide a direct proof of the resolvent smoothing estimate for the special case considered here. 
Let us also mention that Robbiano and Zuily \cite{RobbianoZuily2008}, generalizing a result of Yajima and Zhang \cite{YajimaZhang2004}, proved a smoothing estimate for the propagator similar to that of \cite{Doi2005} under partly more general ($V_0$ can grow superquadratically and $A_0$ superlinearly) and partly more restrictive assumptions (they impose stronger symbol type conditions on $V_0$ and $A_0$). Although the technique of~\cite{RobbianoZuily2008} is simpler than that of~\cite{Doi2005}, it is not directly applicable under our assumptions.

For our purpose, we do not need the full strength of the calculus used in \cite{Doi2005}.
It will be sufficient to use the following metrics, 
\begin{align}\label{metrics g0 g1}
g_0=\rd x^2+\frac{\rd \xi^2}{\langle X\rangle ^2},\quad
g_1=\frac{\rd x^2}{\langle x\rangle ^2}+\frac{\rd \xi^2}{\langle \xi\rangle ^2},
\end{align}
where $X=(x,\xi)\in \R^{2n}$. We refer to Appendix \ref{Appendix psdos} for more details about the calculus we use here and for the notation.

\begin{lemma}\label{lemma commutator estimate}
Assume (A1)--(A2). Then there exists $\lambda\in S_1(1,\langle x\rangle,g_0)$ such that the following hold.
\begin{enumerate}
\item There exist constants $C,c>0$ such that
\begin{align*}
-\{|\xi|^2,\lambda(x,\xi)\}\geq c\langle x\rangle^{-1-\mu}e_{1/2}(x,\xi)^2-C,\quad (x,\xi)\in \R^{2n}.
\end{align*}
\item $[P_0,\lambda^W]-\frac{1}{\I}\{|\xi|^2,\lambda\}^w\in \Op^W(S(1,g_0))$.
\item We have the positive commutator estimate
\begin{align}\label{eq. positive commutator}
\|\langle x\rangle^{-\frac{1+\mu}{2}}E_{1/2} u\|_{L^{2}}^2\lesssim \langle -\I[P_0,\lambda^W]u,u\rangle+\|u\|_{L^2}^2.
\end{align}
\end{enumerate}
\end{lemma}

\begin{proof}
(1) The claim follows from \cite[Lemma 8.3]{Doi2005}. We give a sketch of the proof for the simpler case considered here. Let $\psi,\chi\in C^{\infty}(\R)$ be such that $0\leq \psi,\chi\leq 1$, $\supp(\psi)\subset [1/4,\infty)$, $\psi(t)=1$ for $t\in [1/2,\infty)$, $\psi'\geq 0$, $\supp(\chi)\subset (-\infty,1]$ and $\chi(t)=1$ for $t\leq 1/2$. Further, set $\psi_+(t)=\psi(t)$, $\psi_-(t)=\psi(-t)$ and
\begin{align*}
\psi_0(t)=1-\psi_+(t)-\psi_-(t),\quad \psi_1(t)=\psi_-(t)-\psi_+(t)=-\sgn(t)\psi(|t|).
\end{align*}
We define the function $\lambda:\R^{2n}\to\R$ by
\begin{align}\label{lambda}
-\lambda=\left( \theta\psi_0(\theta)-(M_0-\langle a\rangle^{-{\mu}})\psi_1(\theta)\right)\chi(r),
\end{align}
where $M_0>2$ is a constant, $a(x,\xi)=\frac{x\cdot\xi}{\langle\xi\rangle}$, $\theta(x,\xi)=\frac{a(x,\xi)}{\langle x\rangle}$ and $r(x,\xi)=\frac{\langle x\rangle}{\langle\xi\rangle}$. The claim that $\lambda\in S_1(1,\langle x\rangle,g_0)$ follows from the fact that $\langle x\rangle\leq \langle\xi\rangle$ on the support of $\chi(r)$.
We write $h_0(\xi)=|\xi|^2$ and denote by $H_{h_0}=2\xi\cdot\nabla_x$ the corresponding Hamiltonian vector field. Observe that on the support of $\psi_0(\theta)$, we have $\theta\leq 1/2$. This implies that
\begin{align}\label{Hh0theta}
-H_{h_0}\theta=\frac{2}{\langle x\rangle}\left(\frac{|\xi|^2}{\langle \xi\rangle}-\frac{x\cdot\xi}{\langle x\rangle}\theta\right)\geq \frac{\langle \xi\rangle}{\langle x\rangle}-2.
\end{align}
It can then be shown that
\begin{align*}
-H_{h_0}\lambda
&=\left((H_{h_0}\theta)\psi_0(\theta)+\mu\langle a\rangle^{-\mu-2}|a|(H_{h_0}a)\psi_1(\theta)\right)\chi(r)\\
&+(H_{h_0}\theta)(M_0-\langle a\rangle^{-\mu}-|\theta|)(\psi_+'(\theta)-\psi_-'(\theta))\chi(r)\\
&+\left(\theta\psi_0(\theta)-(M_0-\langle a\rangle^{-\mu})\psi_1(\theta)\right)\chi'(r)(H_{h_0}r)\\
  &\geq \left((H_{h_0}\theta)\psi_0(\theta)+\mu\langle a\rangle^{-\mu-2}|a|(H_{h_0}a)\psi_1(\theta)\right)\chi(r)-C_1\\
&\geq c_1\left(\langle\xi\rangle\langle x\rangle^{-1}\psi_0(\theta)+\langle  x\rangle^{-\mu-1}\langle\xi\rangle\psi(|\theta|)\right)\chi(r)-C_2\\  
&\geq c_2\langle x\rangle^{-1-\mu} e_{1/2}(x,\xi)^2-C_3.
\end{align*}
In particular, we used that $\psi_+'(t)-\psi_-'(t)\geq 0$ and $\psi_0(t)+\psi(|t|)=1$.

(2) We have, by slight abuse of notation,
\begin{align*}
[P_0,\lambda^W]-\frac{1}{\I}\{P_0,\lambda\}^W=\frac{1}{\I}A+\frac{1}{\I}B
\end{align*}
where
\begin{align*}
A=\I(P_0\#\lambda-\lambda\#P_0)^W-\{P_0,\lambda\}^W,\quad B=\{P_0,\lambda\}^W-\{|\xi|^2,\lambda\}^W.
\end{align*}
Lemma \ref{Lemma 3.4 Doi} and Proposition \ref{Proposition membership to symbol classes} imply that $A,B\in \Op^W(S(1,g_0))$.

(3) follows from (1)--(2) together with Corollary \ref{corollary boundedness on L2}, Theorem \ref{theorem sharp Garding} and the calculus for adjoints (Proposition \ref{Proposition selfadjointness of Weyl quantization}) and compositions (Theorem \ref{theorem composition}). 
\end{proof}

To state the following theorem it will be convenient to introduce the spaces $Y\supset X$ and $Y'\subset X'$ with norms
\begin{align}
\|u\|_Y&:=\|u\|_{L^2}+\|\langle x\rangle^{-\frac{1+\mu}{2}}E_{1/2} u\|_{L^2},\label{definition of Y}\\
\|f\|_{Y'}&:=\inf_{f=f_1+f_2}\left(\|f_1\|_{L^2}+\|\langle x\rangle^{\frac{1+\mu}{2}}E_{-1/2} f_2\|_{L^2}\right).
\end{align}
Note that $X=Y\cap L^q$ and $X'=Y'+L^{q'}$.

\begin{theorem}[Smoothing estimate]\label{thm smoothing}
Let $\mu>0$, and let $a>0$ be fixed. Then for all $z\in\C$ with $|\im z|\geq a$ and for all $u\in \mathcal{D}(\R^n)$, we have the estimate
\begin{align}\label{eq. smoothing estimate thm}
\|u\|_{Y}\leq C_0\|(P_0-z)u\|_{Y'}.
\end{align}
The constant $C_0$ depends on $n$, $\mu$, $a$, and on finitely many seminorms $C_{\alpha}$ in \eqref{assumptions on A_0} and \eqref{assumptions on V_0}. 
\end{theorem}

\begin{proof} 
Inequality \eqref{eq. smoothing estimate thm} follows from the following four inequalities: For all $u\in \mathcal{D}(\R^n)$, we have
\begin{align}
\|u\|_{L^2}&\leq |\im z|^{-1}\|(P_0-z)u\|_{L^2},\label{eq. smoothing 1}\\
\|u\|_{L^2}&\lesssim |\im z|^{-1/2}\|\langle x\rangle^{\frac{1+\mu}{2}}E_{-1/2} (P_0-z)u\|_{L^2},\label{eq. smoothing 2}\\
\|\langle x\rangle^{-\frac{1+\mu}{2}}E_{1/2}u\|_{L^2}&\lesssim |\im z|^{-1/2}\|(P_0-z)u\|_{L^2},\label{eq. smoothing 3}\\
\|\langle x\rangle^{-\frac{1+\mu}{2}}E_{1/2}u\|_{L^2}&\lesssim\|\langle x\rangle^{\frac{1+\mu}{2}}E_{-1/2}(P_0-z)u\|_{L^2}.\label{eq. smoothing 4}
\end{align}
Again, \eqref{eq. smoothing 1} immediately follows from \eqref{trivial L2L2 estimate with imaginary part}.
Inequality \eqref{eq. smoothing 2} follows from \eqref{eq. smoothing 3} by a duality argument that we shall postpone to the end of the proof. It remains to prove \eqref{eq. smoothing 3} and \eqref{eq. smoothing 4}.
%
To this end we use \eqref{eq. positive commutator}. Note that we may replace $P_0$ by $P_0-z$ in \eqref{eq. positive commutator} since the commutator with $z\in\C$ is zero. 
Since $\lambda\in S(1,g_0)$, the $L^2$-boundedness of such symbols (Corollary \ref{corollary boundedness on L2}) and the Cauchy-Schwarz inequality yield the estimate
\begin{align*}
\|u\|_{Y}^2&\lesssim\langle -\I[P_0-z,\lambda^W]u,u\rangle+\|u\|_{L^2}^2\\ &\leq
2\|\lambda^Wu\|_{L^{2}}\left(\|(P_0-z) u\|_{L^{2}}+|\im z|\|u\|_{L^2}\right)+\|u\|_{L^2}^2\\
&\leq (4C_{\lambda}|\im z|^{-1}+|\im z|^{-2})\|(P_0-z) u\|_{L^{2}}^2\\
&\lesssim |\im z|^{-1}\|(P_0-z) u\|_{L^{2}}^2
\end{align*}
where $C_{\lambda}:=\|\lambda^W\|_{\mathcal{B}(L^2)}$. Note that $\lambda^W$ is self-adjoint by Proposition \ref{Proposition selfadjointness of Weyl quantization}. In the last inequality, we also used~\eqref{trivial L2L2 estimate with imaginary part} again. This proves~\eqref{eq. smoothing 3}.

Similarly, \eqref{eq. positive commutator} and duality of $Y,Y'$ yield 
\begin{align*}
\|u\|_{Y}^2&\lesssim\langle -\I[P_0-z,\lambda^W]u,u\rangle+\|u\|_{L^2}^2\\ 
&\leq 2\|\lambda^Wu\|_{Y}\|(P_0-z) u\|_{Y'}+2|\im z|\|\lambda^Wu\|_{L^2}\|u\|_{L^2}+\|u\|_{L^2}^2\\
&\leq \epsilon \|\lambda^Wu\|_{Y}^2+\epsilon^{-1}\|(P_0-z) u\|_{Y'}^2+(1+2C_{\lambda}|\im z|)\|u\|_{L^2}^2\\
&\leq \epsilon (C_{\lambda}')^2 \|u\|_{Y}^2+\epsilon^{-1}\|(P_0-z) u\|_{Y'}^2+(1+2C_{\lambda}|\im z|)\|u\|_{L^2}^2
\end{align*}
for any $\epsilon>0$. In addition to the $L^2$-boundedness of $\lambda^W$, we used that the commutator $[\langle x\rangle^{-\frac{1+\mu}{2}}E_{1/2}, \lambda^W]$ is $L^2$-bounded (Corollary \ref{corollary boundedness on L2}), being in $\Op^W(S(1,g_1))$ by Corollary \ref{Corollary commutator} and Proposition \ref{Proposition membership to symbol classes}. This implies that $C_{\lambda}':=\|\lambda^W\|_{\mathcal{B}(Y)}<\infty$. Hiding the term with $\epsilon$ in the above inequality on the left, we get
\begin{align}\label{eq. double smoothing}
\|u\|_{Y}^2\lesssim \|(P_0-z) u\|_{Y'}^2+|\im z|\|u\|_{L^2}^2.
\end{align}
Combining \eqref{eq. smoothing 2} and \eqref{eq. double smoothing}, we get \eqref{eq. smoothing 4}.

We now provide the details of the duality argument leading to \eqref{eq. smoothing 2}. From \eqref{eq. smoothing 1}, \eqref{eq. smoothing 3}, we see that
\begin{align*}
\|(P_0-z)^{-1}f\|_{Y}\lesssim|\im z|^{-1/2}\|f\|_{L^2},\quad \mbox{for all  } f\in (P_0-z)\mathcal{D}(\R^n)\subset L^2(\R^n).
\end{align*} 
Since $(P_0-z)^{-1}$ is $L^2$-bounded and $\mathcal{D}(\R^n)$ is a core\footnote{This is equivalent to the essential selfadjointness of $P_0$ on $\mathcal{D}(\R^n)$.} for $P_0$, the set $(P_0-z)\mathcal{D}(\R^n)$ is dense in $L^2(\R^n)$ \cite[Problem III.5.19]{Kato1966}. Therefore, 
\begin{align*}
(P_0-z)^{-1}\in \mathcal{B}(L^2,Y) \mbox{  with  } \|(P_0-z)^{-1}\|_{ \mathcal{B}(L^2,Y)}\lesssim|\im z|^{-1/2}.
\end{align*}
Let $f\in L^2(\R^n)$, $g\in \mathcal{D}(\R^n)$. Then
\begin{align*}
|\langle f,g\rangle|\leq \|(P_0-\overline{z})^{-1}f\|_{Y}\|(P_0-z)g\|_{Y'}\lesssim |\im z|^{-1/2}\|f\|_{L^2}\|(P_0-z)g\|_{Y'}
\end{align*}
Taking the supremum over all $f\in L^2(\R^n)$ with $\|f\|_{L^2}=1$, we arrive at
\begin{align*}
\|g\|_{L^2}\lesssim|\im z|^{-1/2}\|(P_0-z)g\|_{Y'},
\end{align*}
proving \eqref{eq. smoothing 2}.
\end{proof}


\section{Proof of Theorem \ref{thm resolvent estimate X'X}}\label{section Proof main theorem, general case}
We first prove Theorem \ref{thm resolvent estimate X'X} for $P_0$ and for $\widetilde{P}_0$.

\begin{proof}[Proof of Theorem \ref{thm resolvent estimate X'X} for $P_0$]
Since $X=Y\cap L^q$ and $X'=Y'+L^{q'}$, inequality \eqref{eq. resolvent estimate main theorem} is equivalent to the following four inequalities: For all $u\in \mathcal{D}(\R^n)$, we have
\begin{align}
\|u\|_{Y}&\lesssim\|(P_0-z)u\|_{Y'},\label{9 inequalities 1}\\
\|u\|_{Y}&\lesssim\|(P_0-z)u\|_{L^{q'}},\label{9 inequalities 2}\\
\|u\|_{L^q}&\lesssim\|(P_0-z)u\|_{Y'},\label{9 inequalities 3}\\
\|u\|_{L^q}&\lesssim\|(P_0-z)u\|_{L^{q'}}\label{9 inequalities 4}.
\end{align}
We have already proved \eqref{9 inequalities 1} and \eqref{9 inequalities 4} in Theorem \ref{thm smoothing} and Theorem \ref{thm LpLp'}, respectively. Moreover, \eqref{9 inequalities 3} follows from \eqref{9 inequalities 2} by a duality argument. Since it is a bit more involved than the previous one, we relegate its proof to Appendix \ref{appendix proof duality argument}. It remains to prove~\eqref{9 inequalities 2}. Here we use that $\lambda^W\in\mathcal{B}(L^q)$ since $1<q<\infty$, see Corollary \ref{Corollary Lp boundedness}. Denoting $C_{\lambda,q}:=\|\lambda^W\|_{\mathcal{B}(L^q)}$, it then follows from \eqref{eq. positive commutator} that
\begin{align*}
\|u\|_{Y}^2&\lesssim\langle -\I[P_0-z,\lambda^W]u,u\rangle+\|u\|_{L^2}^2\\ 
&\leq 2\|\lambda^Wu\|_{L^q}\|(P_0-z)u\|_{L^{q'}}+2|\im z|\|\lambda^Wu\|_{L^2}\|u\|_{L^2}+\|u\|_{L^2}^2\\
&\leq \|\lambda^Wu\|_{L^q}^2+\|(P_0-z)u\|_{L^{q'}}^2+(1+2C_{\lambda}|\im z|)\|u\|_{L^2}^2\\
&\leq C_{\lambda,q}^2\|u\|_{L^q}^2+\|(P_0-z)u\|_{L^{q'}}^2+(1+2C_{\lambda}|\im z|)\|u\|_{L^2}^2\\
&\lesssim\|(P_0-z)u\|_{L^{q'}}^2
\end{align*}
where we used \eqref{eq. Lq'Lq in thm} in the last step.
\end{proof}

\begin{proof}[Proof of Theorem \ref{thm resolvent estimate X'X} for $\widetilde{P}_0$]
An inspection of the previous proofs shows that we may add the perturbation $W\in L^{\infty}$ to the operator $P_0$, without any smallness assumption on the norm. This is because our estimates control $L^2$-norms. 

For the reader's convenience, we provide the argument for the first part of the proof of Theorem \ref{thm LpLp'}, i.e.\
for the spectral projection estimate \eqref{O(1) bound spectral projections}.
To this end, we observe that the analogue of the energy estimate \eqref{eq. energy estimate},
\begin{align*}
\|(\I\partial_t-\widetilde{P}_0)\e^{-\I t k} \widetilde{\Pi}_{[k,k+1]}u\|_{L^{\infty}_t([0,T];L^2_x)}\leq \|u\|_{L^2_x},
\end{align*}
holds for the spectral projections $\widetilde{\Pi}_{[k,k+1]}$ of $\widetilde{P}_0$, by selfadjointness.
We write 
\begin{align*}
(\I\partial_t-P_0)\e^{-\I t k} \widetilde{\Pi}_{[k,k+1]}u=(\I\partial_t-\widetilde{P}_0)\e^{-\I t k} \widetilde{\Pi}_{[k,k+1]}u
+\e^{-\I t k} W\widetilde{\Pi}_{[k,k+1]}u
\end{align*}
and apply the Strichartz estimates \eqref{Strichartz estimates} with $(p_1,q_1)=(p,q)$ and $(p_2,q_2)=(\infty,2)$.
This yields
\begin{align*}
\| \widetilde{\Pi}_{[k,k+1]}u\|_{L^q}\lesssim (1+\|W\|_{L^{\infty}})\|u\|_{L^2}.
\end{align*}
Similarly, one can show that all the previous inequalities for $P_0$ continue to hold for $\widetilde{P}_0$ with the same modification of the constant.
\end{proof}

To treat the general case, we write $P=\widetilde{P}_0+L$, where
\begin{align}\label{L}
 L=-2\I A_1\cdot \nabla-\I(\nabla\cdot A_1)+2 A_0\cdot A_1+A_1^2+V_1.
\end{align}

\begin{lemma}\label{L bounded}
Under the assumptions (A3)--(A5), we have $ L\in\mathcal{B}(X,X')$, with
$\|L\|_{\mathcal{B}(X,X')}\leq C_L\epsilon$ for all $\epsilon\in [0,1]$. The constant $C_L$ is independent of $\epsilon$ and depends only on $n$, $\mu$, $\delta$, $\delta'$ and on $\|\langle x\rangle^{-1}A_0\|_{L^{\infty}}$.
\end{lemma}

For the proof of Lemma \ref{L bounded} the following propositions will be used.

\begin{proposition}\label{commutator bound A1 Taylor}
Let $s\leq 1$ and $f\in \Lip(\R^n)$.
Then $[f,\langle D\rangle^{s}]\in\mathcal{B}(L^2)$.
\end{proposition}

\begin{proof}
\cite[Proposition 4.1.A]{MR1121019}. 
\end{proof}

\begin{proposition}\label{commutator bound A1 Schlag et al}
Let $s\leq 1/2$ and $f\in \dot{W}^{\frac{1}{2},2n}(\R^n)$.
Then $[f,\langle D\rangle^{s}]\in\mathcal{B}(L^2)$.
\end{proposition}

\begin{proof}
\cite[Lemma 2.2]{ErdoganEtAl2009}.
\end{proof}

\begin{proof}[Proof of Lemma \ref{L bounded}]
Without loss of generality we may assume that $\epsilon=1$; the dependence of the bound on $\epsilon$ follows by scaling.
We start with the estimate
\begin{align*}
\|L u\|_{X'}&\leq 2\|\langle x\rangle^{\frac{1+\mu}{2}}E_{-1/2} A_1\cdot\nabla u\|_{L^2}+\|(\nabla\cdot A_1)u\|_{L^2}
\\
&+2\|A_0\cdot A_1 u\|_{L^2}+\|A_1^2 u\|_{L^2}+\|V_1 u\|_{L^{q'}}.
\end{align*}
We immediately see from H\"older's inequality that
\begin{align*}
\|(\nabla\cdot A_1)u\|_{L^2}+2\|A_0\cdot A_1 u\|_{L^2}+\|A_1^2 u\|_{L^2}+\|V_1 u\|_{L^{q'}}\\
\leq
(3+2\|\langle x\rangle^{-1}A_0\|_{L^{\infty}})\epsilon\|u\|_{X}.
\end{align*}
It remains to prove 
\begin{align}\label{only hard part in perturbation argument}
\|\langle x\rangle^{\frac{1+\mu}{2}}E_{-1/2} A_1\cdot\nabla u\|_{L^2}\lesssim\|u\|_{X}.
\end{align}
We set $\widetilde{A_1}(x):=\langle x\rangle^{1+\mu}A_1(x)$. Then, \eqref{only hard part in perturbation argument} would follow from 
\begin{align*}
\langle x\rangle^{\frac{1+\mu}{2}}E_{-1/2}\langle x\rangle^{-(1+\mu)} \widetilde{A_1}\cdot\nabla E_{-1/2}\langle x\rangle^{\frac{1+\mu}{2}}\in\mathcal{B}(L^2).
\end{align*}
Writing 
\begin{align*}
\langle x\rangle^{\frac{1+\mu}{2}}E_{-1/2}&=\left(\langle x\rangle^{\frac{1+\mu}{2}}E_{-1/2}\langle D\rangle^{1/2}\langle x\rangle^{-\frac{1+\mu}{2}}\right)\langle x\rangle^{\frac{1+\mu}{2}}\langle D\rangle^{-1/2},\\
E_{-1/2}\langle x\rangle^{\frac{1+\mu}{2}}&=\langle D\rangle^{-1/2}\langle x\rangle^{\frac{1+\mu}{2}}\left(\langle x\rangle^{-\frac{1+\mu}{2}}\langle D\rangle^{1/2}E_{-1/2}\langle x\rangle^{\frac{1+\mu}{2}}\right),
\end{align*}
and using the $L^2$-boundedness of the operators in brackets (a consequence of Corollary \ref{corollary boundedness on L2} and Proposition \ref{Proposition membership to symbol classes}), it remains to prove that
\begin{align}\label{only hard part in perturbation argument tilde}
B:=\langle x\rangle^{\frac{1+\mu}{2}}\langle D\rangle^{-1/2}\langle x\rangle^{-(1+\mu)} \widetilde{A_1}\cdot\nabla \langle D\rangle^{-1/2}\langle x\rangle^{\frac{1+\mu}{2}}\in\mathcal{B}(L^2).
\end{align}
After some commutations, we see that 
\begin{align*}
B=B_0+B_1+B_2+B_3,
\end{align*}
where 
\begin{align*}
B_0&:=\langle D\rangle^{-1/2}\widetilde{A}_1\cdot\nabla\langle D\rangle^{-1/2},\\
B_1&:=\langle x\rangle^{-\frac{1+\mu}{2}}[\langle D\rangle^{-1/2},\langle x\rangle^{\frac{1+\mu}{2}}]\widetilde{A}_1\cdot\nabla\langle D\rangle^{-1/2},\\
B_2&:=\langle x\rangle^{-\frac{1+\mu}{2}}\langle D\rangle^{-1/2}\widetilde{A}_1\cdot[\nabla\langle D\rangle^{-1/2},\langle x\rangle^{\frac{1+\mu}{2}}],\\
B_3&:=\langle x\rangle^{\frac{1+\mu}{2}}[\langle D\rangle^{-1/2},\langle x\rangle^{-(1+\mu)}]\widetilde{A}_1\cdot\nabla\langle D\rangle^{-1/2}\langle x\rangle^{\frac{1+\mu}{2}}.
\end{align*}
Since $\widetilde{A}_1\in L^{\infty}$ (recall that $0<\mu\leq \delta$), it is immediate that $B_2$ is bounded (again a consequence of Corollary \ref{corollary boundedness on L2} and Proposition \ref{Proposition membership to symbol classes}).
Using the identity
\begin{align}\label{commutator identity}
[T^{-1},S]=-T^{-1}[T,S]T^{-1},
\end{align}
with $T=\langle D\rangle^{1/2}$ and $S=\langle x\rangle^{\frac{1+\mu}{2}}$
we can write
\begin{align*}
B_1=-\langle x\rangle^{-\frac{1+\mu}{2}}\langle D\rangle^{-1/2}[\langle D\rangle^{1/2},\langle x\rangle^{\frac{1+\mu}{2}}]B_0.
\end{align*}
Hence, if $B_0$ is bounded, then so is $B_1$. A similar calculation shows that if $B_0$ is bounded, then $B_3$ is bounded. To prove the boundedness of $B_0$, we commute again, using \eqref{commutator identity}, to see that
\begin{align*}
B_0=-\langle D\rangle^{-1/2}[\langle D\rangle^{1/2},\widetilde{A}_1]\langle D\rangle^{-1/2}\cdot\nabla\langle D\rangle^{-1/2}+\widetilde{A}_1\langle D\rangle^{-1/2}\cdot\nabla\langle D\rangle^{-1/2}.
\end{align*}
Clearly, the second term is bounded. The first term is bounded by Proposition~\ref{commutator bound A1 Taylor} (if (A4a) is assumed) or Proposition~\ref{commutator bound A1 Schlag et al} (if (A4b) is assumed). 
\end{proof}

\begin{proof}[Proof of Theorem \ref{thm resolvent estimate X'X} in the general case]
By what we have already proved, Theorem~\ref{thm resolvent estimate X'X} holds for $\widetilde{P}_0$. Lemma \ref{L bounded} then yields that for $u\in \mathcal{D}(\R^n)$, we have
\begin{align*}
\|(P-z)u\|_{X'}\geq \|(P_0-z)u\|_{X'}-\|Lu\|_{X'}\geq \left(\frac{1}{C_0}-C_L\epsilon\right)\|u\|_X\geq \frac{1}{2C_0}\|u\|_X,
\end{align*}
provided $\epsilon\leq (2C_LC_0)^{-1}$; here, $C_0$, $C_L$ are the constants in Theorem \ref{thm resolvent estimate X'X} and Lemma~\ref{L bounded}, respectively.
\end{proof}

\section{Application to complex-valued potentials}\label{section complex-valued potentials}

In this Section we use the resolvent estimate to find upper bounds on the location of (complex) eigenvalues 
for Schr\"odinger operators with complex-valued potentials. For Schr\"odinger operators $-\Delta+V$ with decaying but singular potentials $V$, results of this type have been established e.g.\ in \cite{AAD01} in one dimension and in \cite{Frank11,FrankSimon2015,FanelliKrejcikVega2015} in higher dimensions. Estimates for sums of eigenvalues of the Schr\"odinger operator with constant magnetic field perturbed by complex electric potentials were obtained in \cite{Sambou2014}. The conditions on the potential there are much more restrictive than ours and the results (when applied to a single eigenvalue) are considerably weaker. To our knowledge, our eigenvalue estimates are the first with a complex-valued magnetic potential.

For definiteness, we assume that $P_0$ is either the harmonic oscillator or the Schr\"odinger operator with constant magnetic field (called the Landau Hamiltonian in quantum mechanics), but other examples could easily be accommodated. Hence, from now on, either
\begin{align}\label{harmonic oscillator}
P_0=-\Delta+|x|^2,\quad x\in\R^n \quad(\mbox{Harmonic Oscillator}),
\end{align}
or, for $n$ even and $B_0>0$,
\begin{align}\label{constant magnetic field}
P_0=
\sum_{j=1}^{n/2}\left[\left(-\I\partial_{x_j}-\frac{B_0}{2} y_j\right)^2+\left(-\I\partial_{y_j}+\frac{B_0}{2}x_j\right)^2\right]
\quad\mbox{(Landau Hamiltonian)},
\end{align}
where in the case \eqref{constant magnetic field} we denoted the independent variable by $(x,y)=z\in\R^{n}$.
In the mathematical literature, \eqref{harmonic oscillator} is also called the Hermite operator and \eqref{constant magnetic field} is known as the twisted Laplacian. The spectra of these operators can be computed explicitly to be
\begin{align*}
\sigma(P_0)=2\N+m(n),
\end{align*}
where $m(n)=n$ in the case of \eqref{harmonic oscillator} and $m(n)=n/2$ in the case of \eqref{constant magnetic field}. 

In the notation of \eqref{P_0}, the harmonic oscillator \eqref{harmonic oscillator} corresponds to $P_0$ with $V_0(x)=|x|^2$, $A_0(x)=0$, while the Landau Hamiltonian \eqref{constant magnetic field} corresponds to $P_0$ with $V_0(z)=0$ and
\begin{align*}
A_0(z)=\frac{B_0}{2}(-y_1,x_1,\ldots,-y_{n/2},x_{n/2}),\quad z=(x,y)\in\R^n.
\end{align*}
As before, we allow a real-valued bounded potential $W\in L^{\infty}(\R^n,\R)$ in the definition of the unperturbed operator $\widetilde{P}_0=P_0+W$.
We now consider a perturbation of $\widetilde{P}_0$ by a \emph{complex-valued} electromagnetic potential $(A_1,V_1)$, that is we consider the Schr\"odinger operator
\begin{align*}
P=(-\I\nabla+A)^2+V=\widetilde{P}_0+L,\quad \dom(P)=\mathcal{D}(\R^n).
\end{align*}
Here, $L$ is given by \eqref{L}. 
We only require a smallness assumption on $A_1$, but not on $V_1$. To be clear, we repeat the assumptions on $A_1,V_1$ at this point.
\begin{enumerate}
\item[($A4_{\C}$)] $A_1\in \Lip(\R^n,\C^n)$ and there exists $\delta>0$ such that
\begin{align}
|\nabla A_1(x)|\lesssim \epsilon\langle x\rangle^{-1-\delta}\quad\mbox{for almost every }x\in\R^n.
\end{align}
Moreover, assume that one of the following additional assumptions hold.
\item[($A4a_{\C}$)] $A_1\in \Lip(\R^n,\C^n)$ and
\begin{align*}
|\nabla A_1(x)|\lesssim \epsilon\langle x\rangle^{-1-\delta}\quad\mbox{for almost every }x\in\R^n.
\end{align*}
\item[($A4b_{\C}$)] There exists $\delta'\in (0,\delta)$ such that $\langle x\rangle^{1+\delta'} A_1\in \dot{W}^{\frac{1}{2},2n}(\R^n;\C^n)$, and we have $\|\langle x\rangle^{1+\delta'}A_1\|_{\dot{W}^{\frac{1}{2},2n}}\lesssim \epsilon$.
\item[($A5_{\C}$)] Assume that $V_1\in L^{r}(\R^n,\C)$ for some $r\in(1,\infty]$ if $n=2$ and $r \in[n/2,\infty]$ if $n\geq 3$.
\end{enumerate}

Let $Q(P_0)$ be the form domain\footnote{$Q(P_0)=D(\mathfrak{q}_0)$ in the notation of Appendix \ref{appendix m-sectorial}. We equip $Q(P_0)$ with the form norm \eqref{form norm}.} of $P_0$,
\begin{align}
Q(P_0)=\set{u\in L^2(\R^n)}{\nabla_{A_0} u\in L^2(\R^n),\, V_0^{1/2}u\in L^2(\R^n)}\label{form domain P0}.
\end{align}
By Lemma \ref{lemma relative form bounded} there exists a unique $m$-sectorial extension of $P$ with the property $\dom(P)\subset Q(P_0)$.  	
By abuse of notation we will still denote this extension by $P$ in the following theorem.

\begin{theorem}\label{thm complex potentials}
Assume that $P_0$ is either the harmonic oscillator \eqref{harmonic oscillator} or the Landau Hamiltonian \eqref{constant magnetic field} (when $n$ is even). Assume that $A_1,V_1$ satisfy $A4_{\C}$--$A5_{\C}$. For $a>0$ fixed, there exists $\epsilon_0>0$ such that whenever $A_1$ satisfies $A4_{\C}$ with $\epsilon<\epsilon_0$, then every eigenvalue $z$ of $P$ with $|\im z|\geq a$ satisfies the estimate
\begin{align}\label{spectral estimate complex potentials}
|\im z|^{1-\frac{n}{2r}}\leq \frac{C_0(1+\|W\|_{L^{\infty}})}{1-\epsilon/\epsilon_0}\|V_1\|_{L^r}.
\end{align}
Here, $C_0$ is the constant in the estimate \eqref{eq. resolvent estimate z dependent}, and $\epsilon_0$ depends on $n,\delta,\delta',\mu,a,B_0,C_0$.
\end{theorem}

\begin{remark}
Since the left hand side of \eqref{spectral estimate complex potentials} is $\geq a^{1-\frac{n}{2r}}$ by assumption, Theorem~ \ref{thm complex potentials} implies that for any  $a>0$ and $0\leq \epsilon<\epsilon_0$, there exists $v_0=v_0(a,\epsilon)$ such that for $\|V\|_{L^r}\leq v_0$, all eigenvalues $z\in\sigma(P)$ are contained in $\set{z\in\C}{|\im z|\leq a}$. 
\end{remark}

\begin{proof}[Proof of Theorem \ref{thm complex potentials}]
Assume that $z\in\C$, with $|\im z|\geq a$, is an eigenvalue of~$P$, i.e.\ there exists $u\in \dom(P)$, with $\|u\|_{L^2}=1$, such that $Pu=zu$. Since $P_0\in \mathcal{B}(Q(P_0),Q(P_0)')$ and $L\in \mathcal{B}(Q(P_0),Q(P_0)')$ by Lemma \ref{lemma relative form bounded} ii), we have
\begin{align}\label{eigenvalue equation in Q(P0)'}
(P_0-z)u=-Lu\quad \mbox{in  }Q(P_0)'.
\end{align}
Since $Q(P_0)\subset X(z)$ densely and continuously, by Lemma \ref{lemma continuous embedding}, we have $u\in X(z)$, $\|u\|_{X(z)}\neq 0$, and \eqref{eigenvalue equation in Q(P0)'} implies that $(P_0-z)u+Lu=0$ in $X(z)'$. By Lemma \ref{L bounded}, $L\in\mathcal{B}(X(z),X(z)')$, and thus $(P_0-z)u\in X(z)'$. This means that 
\begin{align}\label{eigenvalue equation in X(z)'}
(P_0-z)u=-Lu\quad \mbox{in  }X(z)'.
\end{align} 
Then \eqref{eigenvalue equation in X(z)'} and \eqref{eq. resolvent estimate z dependent with inverse on left} yield
\begin{align}\label{lower bound for Lu}
\frac{1}{C_0(1+\|W\|_{L^{\infty}})}\|u\|_{X(z)}\leq \|(P_0-z)u\|_{X(z)'}=\|Lu\|_{X(z)'}.
\end{align}
We estimate the right hand side from above, as in the proof of the general case of Theorem \ref{thm resolvent estimate X'X} in Section \ref{section Proof main theorem, general case}, by
\begin{equation*}\begin{split}
\|Lu\|_{X(z)'}&\leq \left(\epsilon C_{n,\mu,\delta}+2\epsilon(1+B_0)|\im z|^{-1}+|\im z|^{n\left(\frac{1}{2}-\frac{1}{q}\right)-1}\|V_1\|_{L^{r}}\right)\|u\|_{X(z)}\\
&\leq \left(\epsilon C_{n,\mu,\delta,a,B_0}+|\im z|^{n\left(\frac{1}{2}-\frac{1}{q}\right)-1}\|V_1\|_{L^{r}}\right)\|u\|_{X(z)}
\end{split}
\end{equation*}
Together with \eqref{lower bound for Lu}, this implies that
\begin{align*}
\|u\|_{X(z)}\leq C_0(1+\|W\|_{L^{\infty}})\left(\epsilon C_{n,\mu,\delta,a,B_0}+|\im z|^{n\left(\frac{1}{2}-\frac{1}{q}\right)-1}\|V_1\|_{L^{r}}\right)\|u\|_{X(z)}.
\end{align*}
Dividing both sides by $\|u\|_{X(z)}\neq 0$, it follows that any eigenvalue $z$ of $P$ with $|\im z|\geq a$, must satisfy the inequality
\begin{align*}
1\leq C_0(1+\|W\|_{L^{\infty}})\left(\epsilon C_{n,\mu,\delta,a,B_0}+|\im z|^{n\left(\frac{1}{2}-\frac{1}{q}\right)-1}\|V_1\|_{L^{r}}\right).
\end{align*} 
If we set $\epsilon_0=1/(C_0(1+\|W\|_{L^{\infty}}) C_{n,\mu,\delta,a,B_0})$,
then this estimate is equivalent to~\eqref{spectral estimate complex potentials}.
\end{proof}

Instead of the Landau Hamiltonian \eqref{constant magnetic field}, we can also consider the Pauli operator with constant magnetic field. For simplicity, we assume that $n=2$ here, but the general case when $n$ is even can be handled with no additional difficulty. On $\mathcal{D}(\R^2,\C^2)$, the Pauli operator is given by
\begin{align}\label{Pauli}
P=
\begin{pmatrix}
(-\I\nabla+A(x))^2+B(x)&0\\
0&(-\I\nabla+A(x))^2-B(x)
\end{pmatrix}
+W(x)
+V_1(x).
\end{align}
Here, $A=(A^1,A^2)=(A_0^1,A_0^2)+(A_1^1,A_1^2)$ and $B=\partial_{1}A^2-\partial_{2}A^1$. We choose $A_0(z)=\frac{B_0}{2}(-y,x)$ for $z=(x,y)\in\R^2$ and $B_0>0$. Although we could easily allow $W$ and $V_1$ to be matrix-valued potentials, we assume that they are scalar multiples of the identity in $\C^2$.
\begin{corollary}\label{thm Pauli}
Assume that $n=2$ and that $P$ is the Pauli operator \eqref{Pauli}. Assume also that $A_1,V_1$ satisfy $A4_{\C}$--$A5_{\C}$. For $a>0$ fixed, there exists $\epsilon_0>0$ such that whenever $A_1$ satisfies $A4_{\C}$ with $\epsilon<\epsilon_0$, then every eigenvalue $z$ of $P$ with $|\im z|\geq a$ satisfies the estimate
\begin{align*}
|\im z|^{1-\frac{n}{2r}}\leq \frac{C_0(1+B_0+\|W\|_{L^{\infty}})}{1-\epsilon/\epsilon_0}\|V_1\|_{L^r}.
\end{align*}
Here, $C_0$ is the constant in the estimate \eqref{eq. resolvent estimate z dependent}, and $\epsilon_0$ depends on $n,\delta,\delta',\mu,a,B_0,C_0$.
\end{corollary}
\begin{proof}
In view of the direct sum structure of \eqref{Pauli}, the proof reduces to proving \eqref{spectral estimate complex potentials} for eigenvalues $z$ of the Schr\"odinger operators
\begin{align*}
P_{\pm}=(-\I\nabla+A(x))^2\pm B(x)+W(x)+V_1(x).
\end{align*}
The proof is analogous to the one of Theorem \ref{thm complex potentials}.
\end{proof}

\begin{remark}
Note that the reason we were able to remove the smallness assumption on $V_1$, but not on $A_1$ is that the smoothing part of the $X(z)$ norm is $z$-independent (see Appendix \ref{Section A more precise version of Theorem}).
\end{remark}

\appendix

\section{Basic facts about the Weyl calculus}\label{Appendix psdos}

We give a brief outline of the pseudodifferential calculus used in the second part of Section \ref{section Lq'Lq and smoothing estimates}. The facts stated here are a condensation of the more general results contained in \cite[Section 3]{Doi2005}. We restrict ourselves to a degree of generality that is sufficient for the situation considered in the main body of the text. To this end, we follow in part the exposition in \cite[Chapter 1]{NicolaRodino2010}. For generalizations, the reader is referred to \cite{Doi2005,Hoermander2007}.    

\begin{definition}[Weights]
A continuous function $\Phi:\R^{2n}\to (0,\infty)$ is called a \emph{weight}. It is called a \emph{sublinear weight} if
\begin{align}\label{sublinear weight}
1\leq \Phi(X)\lesssim 1+|X|,\quad X\in\R^{2n}.
\end{align}
It is called a \emph{temperate weight} if, for some $s>0$,
\begin{align*}
\Phi(X+Y)\lesssim \Phi(X)(1+|Y|)^s,\quad X,Y\in\R^{2n}.
\end{align*}
\end{definition}

Given two weights $\Phi,\Psi$, let $g$ be the following metric on $\R^{2n}$,
\begin{align}\label{metric Phi Psi}
g=\Phi(X)^{-2}\rd x^2+\Psi(X)^{-2}\rd \xi^2,\quad X=(x,\xi)\in\R^{2n}.
\end{align}

\begin{definition}[Symbol classes]\label{Definition symbol classes}
Let $\Phi$ and $\Psi$ be sublinear, temperate weights, and let $g$ be the metric \eqref{metric Phi Psi}. If $m$ be a temperate weight, we denote by $S(m,g)$ the space
\begin{align*}
S(m,g)&=\set{a\in C^{\infty}(\R^{2n})}{\forall\alpha,\beta\in\N^n\,\exists\, C_{\alpha,\beta}>0\mbox{ s.t.}\\
&\qquad|\partial_x^{\alpha}\partial_{\xi}^{\beta}a(X)|\leq C_{\alpha,\beta}m(X)\Phi(X)^{-|\alpha|}\Psi(X)^{-|\beta|}}.
\end{align*}
The family of seminorms
\begin{align*}
\|a\|_{k,S(m,g)}:=\sup_{|\alpha|+|\beta|\leq k}\sup_{X\in\R^{2n}}|\partial_x^{\alpha}\partial_{\xi}^{\beta}a(X)|m(X)^{-1}\Phi(X)^{|\alpha|}\Psi(X)^{|\beta|},\quad k\in\N,
\end{align*}
defines a Fr\'echet topology on $S(m,g)$.
\end{definition}

\begin{lemma}\label{lemma chiepsilon}
Let $\chi\in C_c^{\infty}(\R^{2n})$ be such that $\chi(x,\xi)=1$ in a neighborhood of the origin in $\R^{2n}$, and let $\Phi,\Psi$ be any sublinear weights. Then the family $\chi_{\epsilon}(x,\xi):=\chi(\epsilon x,\epsilon \xi)$, $0\leq \epsilon\leq 1$, is bounded in $S(1,g)$. Moreover, if $m_0$ is any temperate weight that tends to infinity at infinity, then $\chi_{\epsilon}\to 1$ in $S(m_0,g)$ and $\partial_{\xi}^{\alpha}\partial_x^{\beta}\chi_{\epsilon}\to 0$ in $S(m_0,g)$ for $|\alpha|+|\beta|\geq 1$ as $\epsilon\to 0$. 
\end{lemma}

\begin{proof}
\cite[Lemma 1.1.4]{NicolaRodino2010}. The claim that $\partial_{\xi}^{\alpha}\partial_x^{\beta}\chi_{\epsilon}\to 0$ in $S(m_0,g)$ for $|\alpha|+|\beta|\geq 1$ as $\epsilon\to 0$ is not part of the statement there, but follows easily by inspection of the proof.
\end{proof}

\begin{definition}[Planck function]
Let $\Phi,\Psi$ be two weights. The function
\begin{align}\label{Planck function}
h(X):=\Phi(X)^{-1}\Psi(X)^{-1},\quad X\in\R^{2n}
\end{align}
is called the \emph{Planck function}. We say that $h$ satisfies the \emph{strong uncertainty principle} if there exists $\gamma>0$ such that
\begin{align}\label{strong uncertainty principle}
h(X)\lesssim (1+|X|)^{-\gamma},\quad X\in\R^{2n}.
\end{align}
\end{definition}

\begin{remark}\label{Remark cases symbol classes of interest}
The cases of most interest to us are the following.
\begin{enumerate}
\item If $g=g_0$, then $\Phi(X)=1$, $\Psi(X)=\langle X\rangle$, $h(X)=\langle X\rangle^{-1}$.
\item If $g=g_1$, then $\Phi(X)=\langle x\rangle$, $\Psi(X)=\langle \xi\rangle$, $h(X)=\langle x\rangle^{-1}\langle\xi\rangle^{-1}$. 
\item The case $\Phi(X)=1$, $\Psi(X)=\langle \xi\rangle$, $h(X)=\langle \xi\rangle^{-1}$ corresponds to the standard symbol class $S^m_{1,0}$.
\end{enumerate}
\end{remark}

Note that under the assumption \eqref{sublinear weight}, $h$ always satisfies $h(X)\leq 1$ (the uncertainty principle). 
Under the assumption of the strong uncertainty principle, it is meaningful to speak about the asymptotic expansion of symbols in $a\in S(m,g)$.

\begin{lemma}\label{lemma asymptotic expansion}
Let $(a_j)_{j=0}^{\infty}$ be a sequence of symbols $a_j\subset S(mh^j,g)$. Then there exist $a\in S(n,g)$ (uniquely determined modulo Schwartz functions) such that
\begin{align*}
a\sim\sum_{j=0}^{\infty} a_j
\end{align*}
in the sense that for every $N\in\N$,
\begin{align*}
a-\sum_{j=0}^{N-1}a_j\in S(mh^N,g).
\end{align*}
\end{lemma}

\begin{proof}
\cite[Proposition 1.1.6]{NicolaRodino2010}. 
\end{proof}

\begin{definition}[Weyl quantization]\label{definition Weyl quantization}
Let $a\in S(m,g)$. Then the pseudodifferential operator
\begin{align*}
\Op^W(a)u(x):=a^W(x,D)u(x):=(2\pi)^{-n}\int_{\R^n}\int_{\R^n}\e^{\I(x-y)\cdot \xi}a\left(\frac{x+y}{2},\xi\right)u(y)\rd y\rd\xi,
\end{align*}
initially defined for $u\in\mathcal{S}(\R^n)$, is called the \emph{Weyl quantization} of the symbol $a$. 
We denote the class of (Weyl) pseudodifferential operators with symbols $a\in S(m,g)$ by $\Op^W(S(m,g))$. 
\end{definition}


\begin{proposition}[Adjoint]\label{Proposition selfadjointness of Weyl quantization}
Let $a\in S(m,g)$. Then $(a^W)^*=\overline{a}^W$, in the sense that
\begin{align*}
\langle a^W f,g\rangle=\langle f,\overline{a}^W g\rangle,\quad f,g\in\mathcal{S}(\R^n).
\end{align*}
In particular, $(a^W)^*=a^W$ if and only if $a$ is real-valued.
\end{proposition}

\begin{proof}
See e.g.\ \cite[Proposition 1.2.10]{NicolaRodino2010}.
\end{proof}

\begin{theorem}[Composition]\label{theorem composition}
Let $a\in S(m_1,g)$, $b\in S(m_2,g)$. Then there exists a symbol $c\in S(m_1m_2,g)$, denoted by $c:=a\# b$, such that $a^Wb^W=c^W$.  
Moreover, if the Planck function satisfies the strong uncertainty principle \eqref{strong uncertainty principle}, then we have the asymptotic expansion
\begin{align}\label{composition formula in S(m,g)}
c(x,\xi) \sim\sum_{\alpha,\beta}(-1)^{|\beta|}(\alpha!\beta!)^{-1}2^{-|\alpha+\beta|}\partial_{\xi}^{\alpha}D_x^{\beta}a(x,\xi)\partial_{\xi}^{\beta}D_x^{\alpha}b(x,\xi).
\end{align}
Moreover, the map $(a,b)\mapsto a\# b$ is continuous from $S(m_1,g)\times S(m_2,g)$ into $S(m_1m_2,g)$. 
\end{theorem}

\begin{proof}
See e.g.\ \cite[Theorem 1.2.17]{NicolaRodino2010}.
\end{proof}


\begin{corollary}[Commutators]\label{Corollary commutator}
Let $a\in S(m_1,g)$, $b\in S(m_2,g)$. Then the commutator $[a^W,b^W]=a^Wb^W-b^Wa^W$ is a pseudodifferential operator with Weyl symbol $a\#b-b\#a\in S(mh,g)$.
\end{corollary}

\begin{proposition}\label{prop. continuity on S}
Let $a\in S(m,g)$. Then $a^w(x,D)$ is continuous on $\mathcal{S}(\R^n)$. 
Moreover, the map $S(m,g)\times \mathcal{S}(\R^n)\to \mathcal{S}(\R^n)$ is continuous. 
\end{proposition}

\begin{proof}
\cite[Propositions 1.2.7
]{NicolaRodino2010}.
\end{proof}

\begin{theorem}[Boundedness on $L^2$]\label{theorem boundedness on L2}
Assume that $a\in S(1,g)$ and that the strong uncertainty principle \eqref{strong uncertainty principle} holds. Then $a^W(x,D)$ is bounded on $L^2(\R^n)$. Moreover, the map $S(m,g)\times L^2(\R^n)\to L^2(\R^n)$ is continuous. 
\end{theorem}

\begin{proof}
See e.g.\ \cite[Theorem 1.4.1 and Remark 1.2.6]{NicolaRodino2010}.
\end{proof}

\begin{corollary}\label{corollary boundedness on L2}
If $g=g_0$ or $g=g_1$ and $a\in S(1,g)$. Then $a^W$ is $L^2$-bounded. 
\end{corollary} 

\begin{proof}
It is enough to verify that the strong uncertainty principle \eqref{strong uncertainty principle} holds in these cases, see Remark \ref{Remark cases symbol classes of interest}.
\end{proof}

\begin{theorem}[Sharp G\aa rding inequality]\label{theorem sharp Garding} Let $a\in S(h^{-1},g)$ and $a(X)\geq 0$ for all $X\in\R^{2n}$. Then there exists $C>0$ such that
\begin{align*}
\langle u,a^W(x,D)u\rangle_{L^2}\geq -C\|u\|^2_{L^2},\quad u\in\mathcal{S}(\R^n).
\end{align*}
\end{theorem}

\begin{proof}
See e.g.\ \cite[Theorem 1.7.15]{NicolaRodino2010}
\end{proof}

\begin{theorem}[Boundedness on $L^p$]\label{theorem boundedness on Lp}
Let $a\in S^0_{1,0}$. Then $a^W(x,D)$ is bounded on $L^p(\R^n)$ for $1<p<\infty$. Moreover, there exists $k\in\N$ and $C_{q}>0$ such that $\|a^W\|_{\mathcal{B}(L^p)}\leq C_{q}\|a\|_{k,S(1,g)}$. 
\end{theorem}

\begin{proof}
See \cite[Proposition VI.4]{Stein1993}. Notice that the choice of quantization is immaterial for this result since there exists a symbol $a_L\in S^{1}_{1,0}$ such that $a^W(x,D)=a_L(x,D)$ where $a_L(x,D)$ is the standard (or left) quantization of $a$, see e.g.\ \cite[Remark 1.2.6]{NicolaRodino2010}. The second claim can be verified by inspection of the proof of \cite[Proposition VI.4]{Stein1993}.
\end{proof}

\begin{corollary}\label{Corollary Lp boundedness}
Let $g=g_0$ or $g=g_1$, and let $a\in S(1,g)$. Then $a^W(x,D)$ is bounded on $L^p(\R^n)$ for $1<p<\infty$. Moreover, there exists $k\in\N$ and $C_{q}>0$ such that $\|a^W\|_{\mathcal{B}(L^p)}\leq C_{q}\|a\|_{k,S(1,g)}$. 
\end{corollary}

\begin{proof}
This follows immediately from the inclusion $S(1,g)\subset S^1_{1,0}$. 
\end{proof}

We also need the following special case of the symbol classes used in \cite{Doi2005}.

\begin{definition}
Let $\Phi,\Psi$ be as in Definition \ref{Definition symbol classes} and let $\Phi_0\geq \Phi$ be another temperate weight. We define, for $N\in\N$,
\begin{align*}
S_N(m,\Phi_0,g_0)=\{a\in S(m,g):\partial_x^{\alpha}a\in S(\Phi_0^{-|\alpha|}m,g)\mbox{ for all }|\alpha|\leq N\}.
\end{align*} 
Moreover, set $h_0:=\Phi_0^{-1}\Psi^{-1}$.
\end{definition}

\begin{remark}\label{remark Doi's class}
The subsequent lemma will be used with $g=g_0$ and $\Phi_0(X)=\langle x\rangle$ or with $\Phi_0(X)=\Phi(X)=1$.
Note that if $\Phi_0=\Phi$, then $S_N(m,\Phi_0,g)=S(m,g)$ and $h_0=h$.
\end{remark}

\begin{lemma}\label{Lemma 3.4 Doi}
Let $a_1\in S_1(m_1,\Phi_0(X),g)$, $a_2\in S_1(m_2,\Phi_0(X),g)$. Then
\begin{enumerate}
\item $a_1\#a_2\in S_1(m_1m_2,\Phi_0,g)$;
\item $\{a_1,a_2\}\in S(h_0 m_1m_2,g)$;
\item $a_1\#a_2-a_1a_2-\{a_1,a_2\}/(2\I)\in S(hh_0m_1m_2,g)$;
\item $a_1\#a_2-a_2\#a_1-\{a_1,a_2\}/\I\in S(h^2h_0m_1m_2,g)$.
\end{enumerate}
\end{lemma}

\begin{proof}
This follows from \cite[Lemma 3.4]{Doi2005}.
\end{proof}


\begin{proposition}\label{Proposition membership to symbol classes}
Let $g_0$, $g_1$ be defined by \eqref{metrics g0 g1}. Then the following hold.
\begin{enumerate}
\item $\langle x\rangle^s\in S(\langle x\rangle^s,g_j)$ for $j=0,1$ and $s\in\R$;
\item $e_{s}\in S(\langle X\rangle^s,g_j)$ for $j=0,1$ and $s\in\R$;
\item $\langle \xi\rangle^{s}\in S(\langle x\rangle^{s},g_1)$ for $s\in \R$;
\item $\lambda\in S(1,g_1)\cap S_1(1,\langle x\rangle,g_0)$;
\item $V_0\in S_1(\langle x\rangle^2,\langle x\rangle,g_0)$;
\item $A_0^2\in S_1(\langle x\rangle^2,\langle x\rangle,g_0)$;
\item $A_0(x)\cdot\xi\in S_1(\langle x\rangle\langle X\rangle,\langle x\rangle,g_0)$;
\item $P_0(x,\xi)\in S_1(\langle X\rangle^2,\langle x\rangle,g_0)$.
\end{enumerate}
\end{proposition}

\begin{proof}
This is easily checked.
\end{proof}

\section{Definition of $P$ as an $m$-sectorial operator}\label{appendix m-sectorial}

Here we provide the operator theoretic details that were omitted in Section \ref{section complex-valued potentials}.
For simplicity, we assume in addition to (A1)--(A5) that $V_0\geq 0$. This ensures that the quadratic form of the unperturbed operator is nonnegative and makes the definition of the form sum easier. Note that the assumption is satisfied in the applications in Section \ref{section complex-valued potentials}.
Writing $L$ in \eqref{L} in the form
\begin{align*}
L=-2\I A_1\cdot\nabla_{A_0}-\I(\nabla_{A_0}\cdot A_1)-A_0\cdot A_1+A_1^2+V_1\equiv L_A+V_1,
\end{align*}
we define the quadratic forms
\begin{align}
\mathfrak{p}_0(u)&=\|\nabla_{A_0}u\|^2_{L^2}+\| V_0^{1/2}u\|^2,\label{form of P0}\\
\mathfrak{l}(u)&=\langle L_A u,u\rangle+\langle |V_1|^{1/2} u,V_1^{1/2}u\rangle,\label{form of L}
\end{align}
with domains
\begin{align*}
D(\mathfrak{p}_0)&=\set{u\in L^2(\R^n)}{\nabla_{A_0} u\in L^2(\R^n),\, V_0^{1/2}u\in L^2(\R^n)},\\
D(\mathfrak{l})&=\set{u\in L^2(\R^n)}{\nabla_{A_0} u\in L^2(\R^n),\, |V_1|^{1/2}u\in L^2(\R^n)},
\end{align*}
Here, $\nabla_{A_0}=\nabla+\I A_0(x)$ is the covariant derivative, and $V_1^{1/2}=\e^{\I\varphi}|V_1|^{1/2}$ for $V_1=|V_1|\e^{\I\varphi}$. We also use the magnetic Sobolev spaces
\begin{align*}
H^1_{A_0}(\R^n)&=\set{u\in L^2(\R^n)}{\nabla_{A_0} u\in L^2(\R^n)},\quad
\|u\|_{H^1_{{A_0}}}=\|u\|_{L^2}+\|\nabla_{A_0} u\|_{L^2}.
\end{align*}
We have the continuous and dense\footnote{By \cite[Theorem 7.22]{LiebLoss2001} $\mathcal{D}(\R^n)$ is dense in $H_{A_0}^1(\R^n)$.} embedding $D(\mathfrak{p}_0)\subset H_{A_0}^1(\R^n)$ when $D(\mathfrak{p}_0)$ is equipped with the norm
\begin{align}\label{form norm}
\|u\|_{+1}:=(\|u\|_{L^2}^2+\mathfrak{p}_0(u))^{1/2}.
\end{align}
The content of the following lemma is standard. In order to be self-contained, we give a proof.

\begin{lemma}\label{lemma relative form bounded} Assume Assumptions (A1)--(A5) and that $V_0\geq 0$. Then the following hold.
\begin{itemize}
\item[i)] $\mathfrak{p}_0$ is a closed and nonnegative form.

\item[ii)] $\mathfrak{l}$ is relatively bounded with respect to $\mathfrak{p}_0$.
with relative bound zero. 

\item[iii)] $\mathfrak{p}_0+\mathfrak{l}$, with $D(\mathfrak{p}_0+\mathfrak{l})=D(\mathfrak{p}_0)$, is a closed sectorial form.

\item[iv)] There exists a unique $m$-sectorial operator $P$ associated to the form $\mathfrak{p}_0+\mathfrak{l}$, with the property that $\dom(P)\subset D(\mathfrak{p}_0)$.
\end{itemize}

\end{lemma}

\begin{proof}
i) $\mathfrak{p}_0$ is clearly nonnegative. To prove that it is closed, let $(u_n)_n\subset D(\mathfrak{p}_0)$ be a Cauchy sequence with respect to the form norm $\|\cdot\|_{+,1}$. Since $L^2(\R^n)$ is complete, it follows that there exist $u,v,w\in L^2(\R^n)$ such that 
\begin{align*}
u_n\to u,\quad \nabla_{A_0}u_n\to v,\quad V_0^{1/2}u_n\to w\quad \mbox{in  } L^2(\R^n).
\end{align*} 
Since $\nabla_{A_0}$ and $V_0^{1/2}$ are continuous from $\mathcal{D}'(\R^n)$ to $\mathcal{D}'(\R^n)$,
it follows that $\nabla_{A_0}u=v\in L^2(\R^n)$ and $V_0^{1/2}u=w\in L^2(\R^n)$. This shows that $D(\mathfrak{p}_0)$ is complete, i.e.\ $\mathfrak{p}_0$ is closed.

ii) In view of the assumption $V_0\geq 0$ it is sufficient to prove the Lemma for $V_0=0$; in this case, $D(\mathfrak{p}_0)= H_{A_0}^1(\R^n)$. Also, since $W$ is bounded, it does not affect the relative bound, so we may assume $W=0$ as well. Fix $\epsilon>0$  arbitrarily small.
By Sobolev embedding and the diamagnetic inequality
\begin{align}\label{diamagnetic inequality}
\| |u| \|_{H^1}\leq \|u\|_{H^1_{A_0}},
\end{align}
we have the continuous (and dense) embedding  
\begin{align}\label{magnetic Sobolev embedding}
H^1_{A_0}(\R^n)\subset L^q(\R^n).
\end{align}
H\"older's inequality and \eqref{magnetic Sobolev embedding} then yield the estimate
\begin{align}\label{V1 form bounded}
\||V_1|^{1/2} u\|_{L^2}\leq C \|V_1\|_{L^{\frac{q}{q-2}}}^{1/2}\|u\|_{H^1_{A_0}}=C \|V_1\|_{L^r}^{1/2}\|u\|_{H^1_{A_0}}.
\end{align}
By decomposing $V_1=V_{1,R}+(V_1-V_{1,R})$, with $V_{1,R}=V_1 \mathbf{1}\{x:|V_1(x)|\leq R\}$ and absorbing $V_{1,R}\in L^{\infty}(\R^n)$ into $W$,
we can assume that $C \|V_1\|_{L^r}\leq \epsilon/2$.

Next, we observe that since $A_1\in L^{\infty}(\R^n)$, we have
\begin{align}\label{A1grad form bounded}
|\langle L_A u, u\rangle|\leq \frac{\epsilon}{2}\|\nabla_{A_0}u\|^2_{L^2}+C_{\epsilon}\|A_1\|_{L^{\infty}}^2\|u\|^2_{L^2}
\end{align}
for some $C_{\epsilon}>0$. Altogether, \eqref{V1 form bounded}--\eqref{A1grad form bounded} imply that $D(\mathfrak{l})\subset D(\mathfrak{p}_0)$, and that we have the estimate
\begin{align*}
\mathfrak{l}(u)
\leq C\|u\|^2_{L^2}+\epsilon\, \mathfrak{p}_0(u),
\end{align*}
with a constant $C$ depending on $\epsilon$, $\|A_1\|_{L^{\infty}}$ and $\|V_1\|_{L^r}$.

iii) follows from \cite[Theorem VI.3.4]{Kato1965} and i)-ii).

iv) follows from the first representation theorem \cite[Theorem VI.2.1]{Kato1965}.
\end{proof}

\section{A more precise version of Theorem \ref{thm resolvent estimate X'X}}\label{Section A more precise version of Theorem}

We prove a more precise version of Theorem \ref{thm resolvent estimate X'X} for $P_0$
that takes into account the $z$-dependence of the constant.
It is most convenient to include this dependence in the definition of the spaces $X$ and $X'$, compare Section 4 in \cite{KochTataru2009}. The weighted spaces carry the norms
\begin{align*}
\|u\|_{X(z)}&=|\im z|^{\frac{1}{2}}\|u\|_{L^2}+\|\langle x\rangle^{-\frac{1+\mu}{2}}E_{1/2} u\|_{L^2}+| \im z|^{\frac{1}{2}-\frac{n}{2}\left(\frac{1}{2}-\frac{1}{q}\right)}\|u\|_{L^{q}},\\
\|f\|_{X'(z)}&=\inf_{f=f_1+f_2+f_3}\left(|\im z|^{-\frac{1}{2}}\|f_1\|_{L^2}
+\|\langle x\rangle^{\frac{1+\mu}{2}}E_{-1/2} f_2\|_{L^2}\right.\\
&\left.\qquad\qquad+|\im z |^{\frac{n}{2}\left(\frac{1}{2}-\frac{1}{q}\right)-\frac{1}{2}}\|f_3\|_{L^{q'}}\right).
\end{align*}

\begin{proposition}\label{proposition X is Banach}
$X(z)$ and $X'(z)$ are Banach spaces.
\end{proposition}

\begin{proof}
The dual of a normed space is always complete, so we only need to show that $X(z)$ is complete. We may as well prove this for $|\im z|=1$, i.e.\ for the case $X(z)=X$. We recall that $X=Y\cap L^{q'}$, where $Y$ was defined in \eqref{definition of Y}. The spaces $Y$ and $L^{q'}$ are compatible in the sense that whenever $(u_n)_n\subset X$, $u\in Y$, $v\in L^{q'}$ satisfy $\|u_n-u\|_{Y}\to 0$ and $\|v_n-u\|_{L^{q'}}\to 0$, then $u=v\in X$. This is true because $u=v$ in $\mathcal{D}'$. It follows that $X$ is complete if and only if $Y$ is complete (since $L^{q'}$ is complete). Moreover, $Y$ is complete if and only if the operator $T=\langle x\rangle^{-\frac{1+\mu}{2}}E_{1/2}:L^2(\R^n)\to L^2(\R^n)$ with $\dom(T)=\mathcal{D}(\R^n)$ is closable. To prove the latter, let $(u_n)\subset \mathcal{D}(\R^n)$ be a sequence such that $u_n\to 0$ and $Tu_n\to v$ in $L^2(\R^n)$. We have to show that $v=0$. Indeed, since $T:\mathcal{S}'(\R^n)\to \mathcal{S}'(\R^n)$ is continuous \cite[Theorem 4.16]{Zworski2012}, it follows from the assumptions that $Tu_n\to 0$ in $\mathcal{S}'(\R^n)$. Since $\mathcal{S}(\R^n)$ is dense in $L^2(\R^n)$, it follows that $v=0$ in $L^2(\R^n)$.
\end{proof}

\begin{lemma}\label{lemma continuous embedding}
Assume Assumptions (A1)--(A5) and that $V_0\geq 0$.
Let $\mathfrak{p}_0$ be given by \eqref{form domain P0}. Then we have the dense and continuous embedding $D(\mathfrak{p}_0)\subset X(z)$, where $D(\mathfrak{p}_0)$ is equipped with the form norm $\|\cdot\|_{+1}$.
\end{lemma}

\begin{proof}
Since $D(\mathfrak{p}_0)\subset H^1_{A_0}(\R^n)$ continuously and densely, it is sufficient to prove that the embedding $H_{A_0}^1(\R^n)\subset X(z)$ is continuous and dense.
In view of \eqref{magnetic Sobolev embedding} and the fact that $\mathcal{D}(\R^n)$ is dense in $H_{A_0}^1(\R^n)$ \cite[Theorem 7.22]{LiebLoss2001}, it remains\footnote{Recall that $X(z)$ is defined as the closure of $\mathcal{D}(\R^n)$ in the norm $\|\cdot\|_{X(z)}$.} to prove the estimate
\begin{align}\label{embedding magnetic Sobolev space in smoothing space}
\|\langle x\rangle^{-\frac{1+\mu}{2}}E_{1/2}u\|_{L^2}\lesssim \|u\|_{H^1_{A_0}},\quad u\in \mathcal{D}(\R^n).
\end{align}
By (the proof of) Lemma \ref{lemma commutator estimate} we have an analogue of the commutator estimate~\eqref{eq. positive commutator} for the case $V_0=0$, namely
\begin{align}\label{eq. positive commutator A_0 only}
\|\langle x\rangle^{-\frac{1+\mu}{2}}E_{1/2} u\|_{L^{2}}^2\lesssim \langle -\I[-\Delta_{A_0},\lambda^W]u,u\rangle+\|u\|_{L^2}^2.
\end{align}
By the $L^2$-boundedness of $\lambda^W\in \Op^W(S(1,g_0))$ and of $[\nabla_{A_0},\lambda^W]\in \Op^W(S(1,g_0))$ (Proposition \ref{Proposition membership to symbol classes}, Lemma \ref{Lemma 3.4 Doi} and Corollary \ref{corollary boundedness on L2}), we estimate
\begin{align}\label{commutator is H1A0 bounded}
 \langle -\I[-\Delta_{A_0},\lambda^W]u,u\rangle\leq 2\|\nabla_{A_0}u\|_{L^2}\|\nabla_{A_0}\lambda^Wu\|_{L^2}
 &\lesssim \|u\|_{H^1_{A_0}}.
\end{align}
Combining \eqref{eq. positive commutator A_0 only}--\eqref{commutator is H1A0 bounded}, we get \eqref{embedding magnetic Sobolev space in smoothing space}.
\end{proof}

For technical reasons, we assume the following condition on the spectral projections $\Pi_{[k,k+1]}$ of $P_0$ in $n=2$ dimensions, 
\begin{align}\label{L2Linfty spectral projection estimate in n=2}
\|\Pi_{[k,k+1]}\|_{\mathcal{B}(L^2, L^{\infty})}\lesssim 1,\quad n=2.
\end{align}
This is only used in an interpolation argument, see (4.23) in \cite{KochTataru2009}. 
Note that the estimate does not follow from Strichartz estimates. However, for the harmonic oscillator and the Schr\"odinger operator with constant magnetic field, \eqref{L2Linfty spectral projection estimate in n=2} is known to be true \cite{KochTataru2005Hermite,KochRicci2007}; see also \cite{KochTataruZworski2007,SmithZworski2013} for corresponding results on Schr\"odinger operators with a Riemannian metric, but without magnetic field. We do not whether~\eqref{L2Linfty spectral projection estimate in n=2} is true under the general assumptions \eqref{assumptions on A_0}, \eqref{assumptions on V_0} on $P_0$. In any case, we could do without \eqref{L2Linfty spectral projection estimate in n=2} at the expense of an $\epsilon$-loss in the exponent of the $L^q$-part of the $X(z)$-norm.

\begin{theorem}\label{thm more convenient version of the main theorem}
Assume that $V_0\geq 0$ and that $A_0$, $V_0$ satisfy Assumptions \eqref{assumptions on A_0}, \eqref{assumptions on V_0}. If $n=2$, assume also that \eqref{L2Linfty spectral projection estimate in n=2} holds. Fix $a>0$. Then for any $z\in\C$ with $|\im z|\geq a$, the resolvent $(\widetilde{P}_0-z)^{-1}:L^2(\R^n)\to \dom(P_0)$ extends to a bounded operator in $\mathcal{B}(X(z)',X(z))$, and we have the estimate
\begin{align}\label{eq. resolvent estimate z dependent with inverse on left}
\|(\widetilde{P}_0-z)^{-1}u\|_{\mathcal{B}(X(z)',X(z))}\leq C_0(1+\|W\|_{L^{\infty}}).
\end{align}
The constant $C_0$ depends on $n$, $q$, $\mu$, $a$, and on finitely many seminorms $C_{\alpha}$ in \eqref{assumptions on A_0} and \eqref{assumptions on V_0}. 
\end{theorem}

\begin{proof}
We assume that $W=0$, i.e.\ that $\widetilde{P}_0=P_0$. The general case requires the same argument as in the proof of Theorem \ref{thm resolvent estimate X'X}.
We first prove the estimate
\begin{align}\label{eq. resolvent estimate z dependent}
\|u\|_{X(z)}\lesssim \| (P_0-z)u\|_{X(z)'},\quad u\in \mathcal{D}(\R^n).
\end{align}
The proof differs only slightly from that of Theorem \ref{thm resolvent estimate X'X} for $P_0$. For the convenience of the reader we give a sketch of the proof, highlighting the steps where the $z$-dependence plays a role.

$L^{q'}\to L^q$ estimate: 
For $n\geq 3$, interpolation between \eqref{eq. Lq'L2} with $q=2n/(n-2)$ and \eqref{trivial L2L2 estimate with imaginary part} produces the estimate
\begin{align}\label{eq. Lq'L2 improved by interpolation}
\|(P_0-z)^{-1}\|_{\mathcal{B}(L^{q'},L^2)}\lesssim |\im z|^{\frac{n}{2}\left(\frac{1}{2}-\frac{1}{q}\right)-1}.
\end{align}
For $n=2$, the dual of \eqref{L2Linfty spectral projection estimate in n=2} implies (as in the proof of \eqref{eq. Lq'L2}) that
\begin{align}\label{n=2 L1L2 estimate}
\|(P_0-z)^{-1}\|_{\mathcal{B}(L^{1},L^2)}\lesssim |\im z|^{-1/2}.
\end{align}
Interpolating \eqref{n=2 L1L2 estimate} with \eqref{trivial L2L2 estimate with imaginary part} yields \eqref{eq. Lq'L2 improved by interpolation} in the case $n=2$.
Inequality \eqref{almost resolvent estimate with Strichrtz} can be modified to
\begin{align}\label{z dependent almost Lq'Lq}
\|u\|_{L^{q}}\lesssim |\im z|^{\frac{n}{2}\left(\frac{1}{2}-\frac{1}{q}\right)-\frac{1}{2}}\|u\|_{L^2}+| \im z|^{n\left(\frac{1}{2}-\frac{1}{q}\right)-1}\|(P_0-z)u\|_{L^{q'}}.
\end{align}
To see this, one applies the Strichartz estimates \eqref{Strichartz estimates} to the function $v(x,t)=\e^{-\I t z}u(x)$ with $t$ localized to an interval of size $\mathcal{O}(| \im z|^{-1})$. Combining \eqref{eq. Lq'L2 improved by interpolation} and \eqref{z dependent almost Lq'Lq}, we obtain
\begin{align}\label{z dependent Lq'Lq}
\|u\|_{L^{q}}\lesssim |\im z|^{n\left(\frac{1}{2}-\frac{1}{q}\right)-1}\|(P_0-z)u\|_{L^{q'}}.
\end{align}

$Y'\to Y$ estimate: Denote by $Y(z)$ the space defined by the part of the $X(z)$-norm without the $L^q$-norm. From \eqref{eq. smoothing 1}--\eqref{eq. smoothing 4}, we get
\begin{align}\label{Y(z)'-Y(z) estimate}
\|u\|_{Y(z)}\lesssim \|(P_0-z)u\|_{Y(z)'}.
\end{align}

$L^{q'}\to Y$ estimate: From \eqref{eq. positive commutator}, \eqref{eq. Lq'L2 improved by interpolation}, \eqref{z dependent  Lq'Lq}, \eqref{trivial L2L2 estimate with imaginary part} and since $|\im z|\geq a$, it follows that
\begin{equation}\label{L^q'-Y(z) estimate}
\begin{split}
\|u\|_{Y(z)}^2&\lesssim\langle -\I[P_0-z,\lambda^W]u,u\rangle+|\im z|\|u\|_{L^2}^2\\ 
&\leq 2\|\lambda^Wu\|_{L^q}\|(P_0-z)u\|_{L^{q'}}+2|\im z|\|\lambda^Wu\|_{L^2}\|u\|_{L^2}+|\im z|\|u\|_{L^2}^2\\
&\leq C_{\lambda,q}^2 |\im z|^{-n\left(\frac{1}{2}-\frac{1}{q}\right)+1}\|u\|_{L^{q}}^2+| \im z|^{n\left(\frac{1}{2}-\frac{1}{q}\right)-1}\|(P_0-z) u\|_{L^{q'}}^2\\
&\quad+| \im z|(1+2C_{\lambda})\|u\|_{L^2}^2\\
&\lesssim|\im z|^{n\left(\frac{1}{2}-\frac{1}{q}\right)-1}\|(P_0-z) u\|_{L^{q'}}^2.
\end{split}
\end{equation}

In the third line we also used the Peter Paul inequality
\begin{align*}
2ab\leq \epsilon a^2+\frac{b^2}{\epsilon}\quad \mbox{with   }\epsilon= |\im z|^{-n\left(\frac{1}{2}-\frac{1}{q}\right)+1}.
\end{align*}
The previous estimates \eqref{eq. Lq'L2 improved by interpolation}, \eqref{z dependent Lq'Lq}, \eqref{Y(z)'-Y(z) estimate}, \eqref{L^q'-Y(z) estimate} in conjunction with the trivial $L^2$-estimate \eqref{trivial L2L2 estimate with imaginary part} prove
\eqref{eq. resolvent estimate z dependent}.

Next we prove \eqref{eq. resolvent estimate z dependent with inverse on left}. Consider the norm
\begin{align*}
|||u|||:=\|u\|_{X(z)}+\|(P_0-z)u\|_{X(z)'},\quad u\in \mathcal{D}(\R^n). 
\end{align*}
By \eqref{eq. resolvent estimate z dependent}, we have
\begin{align}\label{triple graph norm bound}
|||u|||\lesssim \|(P_0-z)u\|_{X(z)'}\leq \|(P_0-z)u\|_{L^2}\leq (1+|z|)\|u\|_{P_0},
\end{align}
where $\|u\|_{P_0}=\|u\|_{L^2}+\|P_0 u\|_{L^2}$ is the graph norm of $P_0$. Since $\mathcal{D}(\R^n)$ is a core for $P_0$, it follows from \eqref{triple graph norm bound} that
\begin{align*}
\overline{\mathcal{D}(\R^n)}^{|||\cdot|||}\supset \overline{\mathcal{D}(\R^n)}^{\|\cdot\|_{P_0}}=\dom(P_0).
\end{align*}
This implies, in particular, that \eqref{eq. resolvent estimate z dependent} holds for all $u\in \dom(P)$. Then, since $(P_0-z):\dom(P)\to L^2$ is bijective, it follows that 
\begin{align}\label{resolvent inequality for f in L2}
\|(P_0-z)^{-1}f\|_{X(z)}\lesssim\|f\|_{X(z)'}\quad \mbox{for all   }f\in L^2(\R^n).
\end{align}
Because $X(z)\subset L^2(\R^n)$ is an embedding (and therefore injective), it follows that $ L^2(\R^n)$ is dense in $X(z)'$. Therefore, \eqref{eq. resolvent estimate z dependent with inverse on left} follows from \eqref{resolvent inequality for f in L2} by density.
\end{proof}

\section{Proof of \eqref{9 inequalities 3}}\label{appendix proof duality argument}

We begin with the following lemma.

\begin{lemma}\label{lemma for proof of 9 inequalities 3}
Let $z\in \C\setminus\sigma(P_0)$, and let $u\in Y$ be such that $(P_0-z)u\in L^{q'}(\R^n)$. Then there exist $(u_{\epsilon})_{0\leq \epsilon\leq 1}\subset \mathcal{D}(\R^n)$ such that, as $\epsilon\to 0$,
\begin{itemize}
\item[i)] $u_{\epsilon}\to u$ in $Y$, and 
\item[ii)] $(P_0-z)u_{\epsilon}\to (P_0-z)u$ in $L^{q'}(\R^n)$. 
\end{itemize} 
\end{lemma}

\begin{proof}
The method of proof is similar as that of \cite[C.2.12]{Zworski2012} (the latter is easier since the symbol is quadratic). We first show that there exist $(u_{\epsilon})_{0\leq \epsilon\leq 1}\subset \mathcal{S}(\R^n)$ satisfying i)--ii). To this end, pick $\chi\in \mathcal{D}(\R^{2n})$ with $\chi\equiv 1$ in $B(0,1)$, and set $\chi_{\epsilon}(x,\xi)=\chi(\epsilon x,\epsilon\xi)$. Since $\chi_{\epsilon}^W(x,D)$ maps $\mathcal{S'}$ to $\mathcal{S}$ \cite[Theorem 4.1]{Zworski2012}, we have that $u_{\epsilon}=\chi_{\epsilon}^W(x,D)u\in \mathcal{S}$. By Lemma \ref{lemma chiepsilon} and Theorem \ref{theorem boundedness on Lp}, we have that
\begin{align}\label{uepsilon to u in Lp}
u_{\epsilon}\to u \quad\mbox{in  }L^p(\R^n)
\end{align}
as $\epsilon\to 0$. 
To finish the proof of i), 
it remains to show that
\begin{align*}
\lim_{\epsilon\to 0}\|\langle x\rangle^{-\frac{1+\mu}{2}}E_{1/2}(u-u_{\epsilon})\|_{L^2}=0.
\end{align*}
By \eqref{uepsilon to u in Lp} for $p=2$ and with $u$ replaced by $\langle x\rangle^{-\frac{1+\mu}{2}}E_{1/2}u$, it is sufficient to show that
\begin{align}\label{commutator T,chieps to 0}
\lim_{\epsilon\to 0}\|[\langle x\rangle^{-\frac{1+\mu}{2}}E_{1/2},\chi_{\epsilon}^W]u\|_{L^2}=0.
\end{align}
We use Lemma \ref{lemma chiepsilon} with $m_0(X)=\langle X\rangle^{\rho}$ for an arbitrary $\rho\in(0,1/2)$. Then, since the family (with respect to $0\leq \epsilon\leq 1$) of commutators in \eqref{commutator T,chieps to 0} is contained in a bounded subset of $\Op^W(S(\langle X\rangle^{-1/2+\rho},g_1))$ by Proposition \ref{Proposition membership to symbol classes} and Corollary \ref{Corollary commutator}, it is $L^2$-bounded by Corollary \ref{corollary boundedness on L2}. Moreover, the last part of Theorem \ref{theorem composition} implies that all seminorms of the commutator symbol tend to $0$ as $\epsilon\to 0$. The last statement in Theorem \ref{theorem boundedness on L2} then implies~\eqref{commutator T,chieps to 0}.

To prove ii), by the same argument as before, it is sufficient to show that 
\begin{align}\label{commutator P0,chieps to 0}
\lim_{\epsilon\to 0}\|[P_0,\chi_{\epsilon}^W]u\|_{L^{q'}}=0.
\end{align}
It is sufficient to prove \eqref{commutator P0,chieps to 0} for $u\in\mathcal{S}(\R^n)$, together with the uniform bound
\begin{align}\label{commutator P0,chieps uniform bound}
\|[P_0,\chi_{\epsilon}^W]u\|_{L^{q'}}\leq C\|u\|_{L^{q'}},\quad u\in L^{q'}(\R^n).
\end{align}
We may replace $\chi_{\epsilon}^W$ in \eqref{commutator P0,chieps to 0} by $\chi_{\epsilon}^W-1$. Then \eqref{commutator P0,chieps to 0} follows from Lemma \ref{lemma chiepsilon} and Proposition \ref{prop. continuity on S} for $u\in\mathcal{S}(\R^n)$ since the $L^{q'}$ norm may be estimated by some Schwartz seminorm. 
By Lemma \ref{lemma chiepsilon}, Lemma \ref{Lemma 3.4 Doi} and Remark \ref{remark Doi's class}, we have that $\{[P_0,\chi_{\epsilon}^W]-\{P_0,\chi_{\epsilon}\}^W/(2\I):0<\epsilon<1\}$ is a bounded subset of $\Op^W(S(\langle X\rangle^{-1},g_0))$. Hence, Corollary~\ref{Corollary Lp boundedness} implies that
\begin{align}\label{LpLp commutator bound P0,chi uniform}
\|[P_0,\chi_{\epsilon}^W]u-\{P_0,\chi_{\epsilon}\}^W/(2\I)u\|_{L^{q'}}\leq C\|u\|_{L^{q'}}.
\end{align}
Moreover, since $\epsilon\langle X\rangle=\mathcal{O}(1)$ on the support of $\chi_{\epsilon}$, we have that $\{\{P_0,\chi_{\epsilon}\}:0<\epsilon<1\}$ is a bounded subset of $S(1,g_0)$, and thus, by Corollary \ref{Corollary Lp boundedness} again, 
\begin{align}\label{Poisson bracket bound uniform}
\|\{P_0,\chi_{\epsilon}\}^Wu\|_{L^{q'}}\leq C\|u\|_{L^{q'}}.
\end{align}
Inequalities \eqref{LpLp commutator bound P0,chi uniform}--\eqref{Poisson bracket bound uniform} imply \eqref{commutator P0,chieps uniform bound}.

A a smooth cutoff procedure and a repetition of the above arguments proves the existence of a sequence $(u_{\epsilon})_{0\leq \epsilon\leq 1}\subset \mathcal{D}(\R^n)$ satisfying i)--ii).

\end{proof}

\begin{proof}[{\bf Proof of \eqref{9 inequalities 3}}]
From \eqref{9 inequalities 2} and Lemma \ref{lemma for proof of 9 inequalities 3} it follows that
\begin{align*}
\|u\|_{Y}\lesssim\|(P_0-z)u\|_{L^{q'}},\quad \mbox{for all   } u\in Y\cap (P_0-z)^{-1}(L^2\cap L^{q'}).
\end{align*}
This is equivalent to
\begin{align*}
\|(P_0-z)^{-1}f\|_{Y}\lesssim\|f\|_{L^{q'}},\quad \mbox{for all   } f\in (P_0-z)Y\cap L^2\cap L^{q'}.
\end{align*}
We can omit the intersection with $(P_0-z)Y$ since $\dom(P_0)\subset Y$ (proved below) and $(P_0-z)\dom(P_0)=L^2$ (since $z\in\C\setminus\sigma(P_0)$). 
Since $L^2\cap L^{q'}$ is dense in $L^{q'}$, there is a unique continuous extension $(P_0-z)^{-1}\in\mathcal{B}(L^{q'},Y)$.

Since $\dom(P_0)$ and $Y$ are the closures of $\mathcal{D}(\R^n)$ with respect to the graph norm of $P_0$ and $T=\langle x\rangle^{-\frac{1+\mu}{2}}E_{1/2}$, respectively, the inclusion $\dom(P_0)\subset Y$ would follow if we showed that
\begin{align*}
\|Tu\|_{L^2}\leq a\|P_0 u\|_{L^2}+b\|u\|_{L^2},\quad \mbox{for all  } u\in \mathcal{D}(\R^n),
\end{align*}
for some $a,b>0$. Such an estimate follows from the commutator bound \eqref{eq. positive commutator}.
\end{proof}

\noindent
{\bf Acknowledgements.} 
{\small J.-C. was supported by Schweizerischer Nationalfonds, SNF, through the post doc stipend P300P2\_\_147746. C.K. was supported by NSF grants DMS-1265249 and DMS-1463746. J.-C. would like to to acknowledge the hospitality of the University of Chicago. The authors would like to thank Ari Laptev for suggesting this topic and the referees for valuable comments and suggestions.}

\bibliographystyle{plain}
\bibliography{bibliography}

\end{document}